\documentclass[12pt, oneside]{amsart}
\usepackage{amsthm}
\usepackage{fullpage}
\usepackage[margin=1in]{geometry}
\usepackage{amsmath}
\usepackage{mathrsfs}
\usepackage{amssymb}
\usepackage{subfigure}
\usepackage{verbatim}
\usepackage{caption}
\usepackage{color}
\usepackage[dvipsnames]{xcolor}
\usepackage{enumerate}

\usepackage{array}
\usepackage{graphicx}
\usepackage{subfigure}

\newcommand{\beq}[1]{\begin{equation}\label{#1}}
\newcommand{\eeq}{\end{equation}}
\newcommand{\blem}[1]{\begin{lemma}\label{#1}}
\newcommand{\elem}{\end{lemma}}
\newcommand{\bth}[1]{\begin{theorem}\label{#1}}
\newcommand{\enth}{\end{theorem}}
\newcommand{\brem}[1]{\begin{remark}\label{#1}}
\newcommand{\erem}{\end{remark}}
\newcommand{\Case}[2]{\noindent {\bf Case #1:} \emph{#2}}

\newcommand{\Cov}[1]{\mathscr{#1}}
\newcommand{\powerset}[1]{\operatorname{Pow}(#1)}

\newtheorem{theorem}{Theorem}[section]
\newtheorem{lemma}[theorem]{Lemma}
\newtheorem{prop}[theorem]{Proposition}
\newtheorem{corollary}[theorem]{Corollary}

\theoremstyle{definition}
\newtheorem{defn}{Definition}

\def\HH{H}  

\def\t{{\bf t}}

\begin{document}

\author{Yifan Jing}
\address{%
Department of Mathematics, University of Illinois at Urbana-Champaign, Urbana, IL 61801, USA.}
\email{yifanjing17@gmail.com.}
\author{Alexandr Kostochka}
\address{%
Department of Mathematics, University of Illinois at Urbana-Champaign, Urbana, IL 61801, USA, and Sobolev Institute of Mathematics, Novosibirsk 630090, Russia.}
\email{kostochk@math.uiuc.edu.}
\author{Fuhong Ma}
\address{%
School of Mathematics, Shandong University, Jinan 250100, China.}
\email{mafuhongsdnu@163.com.}
\author{Pongpat Sittitrai}
\address{%
Department of Mathematics, Faculty of Science, Khon Kaen University, 40002, Thailand.}
\email{pongpat\_s@kkumail.com.}
\author{Jingwei Xu}
\address{%
Department of Mathematics, University of Illinois at Urbana-Champaign, Urbana, IL 61801, USA.}
\email{jx6@illinois.edu.}

\thanks{A.K.~was partially supported by NSF grants DMS1600592, by grants 18-01-00353A and 19-01-00682 of the Russian Foundation for Basic Research and
 by Arnold O. Beckman Campus Research Board Award RB20003 of the University of Illinois at Urbana-Champaign.}
\thanks{F.M.~was supported by NNSF grants 11671232 and Shandong Province Natural Science Foundation (ZR2018MA001) of China. Corresponding author.}
\thanks{J.X.~was partially supported   by Arnold O. Beckman Campus Research Board Award RB20003 of the University of Illinois at Urbana-Champaign.}

\title{Defective DP-colorings of sparse multigraphs}

\date{}

\maketitle
\begin{abstract}
DP-coloring (also known as correspondence coloring) is a generalization of list coloring developed recently by Dvo\v r\' ak and Postle.
 We introduce and study $(i,j)$-defective DP-colorings of multigraphs. We concentrate on
sparse multigraphs and consider $f_{DP}(i,j,n)$ --- the minimum number of edges that may have an $n$-vertex  {\em $(i,j)$-critical multigraph}, that is, a multigraph $G$ that has
no  $(i,j)$-defective DP-coloring but whose every proper subgraph has such a coloring. For every $i$ and $j$, we find linear lower bounds
on  $f_{DP}(i,j,n)$ that are exact for infinitely many $n$.
\end{abstract}

\section{Introduction}

\subsection{Defective Coloring}

A \emph{proper $k$-coloring} of a graph $G$ is a partition of $V(G)$ into $k$ independent sets $V_1,\dots,V_k$. 
A \emph{$(d_1, \dots, d_k)$-defective coloring} (or simply  {\em $(d_1, \dots, d_k)$-coloring}) of a graph $G$ is a partition of $V(G)$ into sets $V_1,V_2,\dots,V_k$
such that for every $i\in[k]$, every vertex in $V_i$ has at most $d_i$ neighbors in $V_i$. In particular, a proper $k$-coloring
is a $(0,0,\ldots,0)$-defective coloring. A number of significant results 
 on defective colorings of graphs were obtained in~\cite{Ar1,CCW1,DKMR,EKKOS,EMOP,HW18,KYu,LL66,OOW,VW}.
 
 While it is easy to check whether a graph is $(0,0)$-colorable (i.e., bipartite), for every $(i,j)\neq (0,0)$, it is an NP-complete problem 
 to decide whether
 a graph $G$ has an $(i,j)$-coloring.  In particular, Esperet,
Montassier, Ochem, and Pinlou \cite{EMOP} proved that the problem of verifying whether a given planar graph of girth $9$ has a $(0,1)$-coloring is NP-complete. In view of this, there was a series of papers estimating how sparse can be graphs not admitting $(i,j)$-coloring for given $i$ and $j$, see e.g.
\cite{BIMOR10,BIMR11,BIMR12,BK11,BK14,BKY13,KKZ14,KKZ15}. One of often used measures of sparsity is the {\em maximum average degree},
$mad(G)=\max_{G'\subseteq G}\frac{2|E(G')|}{|V(G')|}$. In this paper we restrict ourselves to coloring with 2 colors.
One of the ways to handle the problem is to study $(i,j)$-{\em critical} graphs, that is,
the graphs that do not have $(i,j)$-coloring but every proper subgraph of which has such a coloring. Let 
$f(i,j,n)$  denote the minimum number of edges in an $(i,j)$-critical $n$-vertex graph. For example, since every acyclic graph is $(0,0)$-colorable, for odd $n$ we have $f(0,0,n)=n$. In the above papers, a number of interesting bounds on $f(i,j,n)$ were proved. In particular, for $j\geq 2i+2$ and also for $(i,j)\in \{(0,1),(1,1)\}$ lower bounds were proved that are exact for infinitely many $n$. 

\subsection{Defective List Coloring}

A \emph{list-assignment} of a graph $G$ is a function $L: V(G)\rightarrow \mathcal{P}(\mathbb{N})$ that assigns to each $v\in V(G)$ a list $L(v)$ of `colors'. $L$ is an \emph{$\ell$-list assignment} if the list of every vertex is of size $\ell$. An \emph{$L$-coloring} of $G$ is a function $\phi : V(G) \to \bigcup_{v\in V(G)} L(v)$ such that $\phi(v)\in L(v)$ for every $v\in V(G)$ and $\phi(u)\neq \phi(v)$ whenever $uv\in E(G)$. A graph $G$ is \emph{$k$-choosable} if $G$ has an $L$-coloring for every $k$-list assignment $L$. The following  notion
 was introduced in~\cite{EH1,S1999} and studied in~\cite{S2000,Woodall1,HS06,HW18}: A {\em $d$-defective list $L$-coloring} 
of $G$ is a function $\phi : V(G) \to \bigcup_{v\in V(G)} L(v)$ such that $\phi(v)\in L(v)$ for every $v\in V(G)$ and every vertex has at most
$d$ neighbors of the same color. If $G$ has a $d$-defective list $L$-coloring from every $k$-list assignment $L$, then it is called
{\em $d$-defective $k$-choosable}. 
As in the case of ordinary coloring, a direction of study is showing that ``sparse" graphs are $d$-defective $k$-choosable.
As mentioned before, in this paper we consider only $k=2$. The best known bounds on maximum average degree that guarantee that a
graph is $d$-defective $2$-choosable are due to Havet and Sereni~\cite{HS06} (a new proof of the lower bound is due to Hendrey and Wood~\cite{HW18}):

\bigskip
\noindent{\bf Theorem A} (\cite{HS06}){\bf.} {\em For every $d\geq 0$, if $mad(G)<\frac{4d+4}{d+2}$, then $G$ is $d$-defective $2$-choosable.
On the other hand, for every $\epsilon>0$, there is a graph $G_\epsilon$ with $mad(G_\epsilon)<4+\epsilon-\frac{ 2d+4}{d^2+2d+2}$ that is
not $(d,d)$-colorable.}

\bigskip

\subsection{Defective DP-Coloring}

In order to solve some problems on list coloring, Dvo\v r\'ak and Postle \cite{DP18} introduced and studied the more general notion of DP-coloring. This notion was extended to multigraphs by Bernshteyn, Kostochka and Pron~\cite{BKP}.

\begin{defn}\label{defn:cover}
		Let $G$ be a multigraph. A \emph{cover} of $G$ is a pair $\Cov{H} = (L, \HH)$, consisting of a graph $\HH$ (called the \emph{cover graph} of $G$) and a function $L \colon V(G) \to \powerset{V(\HH)}$, satisfying the following requirements:
		\begin{enumerate}
			\item the family of sets $\{L(u) \,:\,u \in V(G)\}$ forms a partition of $V(\HH)$;
			\item for every $u \in V(G)$, the graph $\HH[L(u)]$ is complete;
			\item if $E(\HH[L(u), L(v)]) \neq \varnothing$, then either $u = v$ or $uv \in E(G)$;
			\item \label{item:matching} if the multiplicity of an edge $uv \in E(G)$ is $k$, then $\HH[L(u), L(v)]$ is the 
union  of at most $k$  matchings connecting $L(u)$ with $L(v)$. (For simplicity, we only consider, throughout our paper, perfect matching whenever there is an edge between $u$ and $v$.)
		\end{enumerate}
		A cover $\Cov{H} = (L, \HH)$ of $G$ is \emph{$k$-fold} if $|L(u)| = k$ for every $u \in V(G)$.
	\end{defn}

In this paper, we consider only $2$-fold covers and by {\bf graphs} below we always mean {\em multigraphs with no  loops}.

 For a graph $G$ with a cover $\Cov{H} = (L,\HH)$, the set $V(\HH)$ is partitioned into two parts $P$ and $R$ such that for every $v\in V(G)$, $|L(v)\cap P| = |L(v)\cap R| = 1$. The vertices in $P$ are called \emph{poor}, those in $R$ are called \emph{rich}. For every vertex $v\in V(G)$, denote the poor vertex in $L(v)$ by $p(v)$, the rich one by $r(v)$.

   \begin{defn}
		Let $G$ be a graph and $\Cov{H} = (L, \HH)$ be a cover of $G$. An \emph{$\Cov{H}$-map} of $G$   is
	an injection $\phi: V(G)\rightarrow V(\HH)$, such that $\phi(v)\in L(v)$ for every $v\in V(G)$. The subgraph of $\HH$ induced by $\phi(V(G))$ is called the \emph{$\phi$-induced  graph}, denoted by $\HH_{\phi}$. 
	\end{defn}
	

	\begin{defn}[An $(i,j)$-coloring]\label{def3}
		Let $0\leq i\leq j$. Let $G$ be a graph and $\Cov{H}=(L,\HH)$ be its cover. An $\Cov{H}$-map $\phi$
		of $G$   is an \emph{$(i, j)$-defective-DP-coloring of $\HH$}
		if  the degree of every poor vertex in $\HH_\phi$ is at most $i$ and  the degree of every rich vertex in $\HH_\phi$ is at most $j$.
	    We say that $G$ is \emph{$(i,j)$-defective-DP-colorable} if for every $2$-fold cover $\Cov{H} = (L,\HH)$ of $G$, $\HH$ admits an $(i, j)$-defective-DP-coloring.
		\end{defn}
		
For brevity, in the rest of the paper we call an $(i, j)$-defective-DP-coloring simply {\em $(i,j)$-coloring}, and instead of 
``$(i, j)$-defective-DP-colorable" say  ``$(i,j)$-colorable".

\begin{defn}[$(i,j)$-critical graphs]\label{ijcr}
Given $0\leq i\leq j$,  a multigraph $G$ is   \emph{$(i, j)$-critical}, if $G$ is not $(i, j)$-colorable, 
but every proper subgraph of $G$ is. Let $f_{DP}(i,j,n)$ be the minimum number of edges in an $n$-vertex  $(i, j)$-critical multigraph.
\end{defn}

The goal of our paper is to find linear lower bounds for $f_{DP}(i,j,n)$ that are exact for all $i\leq j$ for infinitely many $n$. Since
every not $(i,j)$-colorable graph contains an $(i, j)$-critical subgraph, this will yield best possible bounds on sparseness of graphs 
that provides the existence of $(i,j)$-colorings.

\section{Results}

The goal of this paper is to prove the following extremal result.

\begin{theorem}\label{main}
\begin{enumerate}
        \item If $i = 0$ and $j\geq 1$, then $f_{DP}(0,j,n) \geq n+j$. This is sharp for every $j\geq 1$       and every $n\geq 2j+2$.
        \item If $i\geq1$ and $j\geq 2i+1$, then $f_{DP}(i,j,n)\geq \frac{(2i+1)n-(2i-j)}{i+1}$. This is sharp for each such pair $(i,j)$ for infinitely many $n$.
        \item If $i\geq1$ and $i+2\leq j\leq 2i$, then $f_{DP}(i,j,n)\geq \frac{2jn+2}{j+1}$. This is sharp for each such pair $(i,j)$ for infinitely many $n$.
        \item If $i\geq1$, then $f_{DP}(i,i+1,n)\geq \frac{(2i^2+4i+1)n+1}{i^2+3i+1}$. This is sharp for each  $i\geq 1$ for infinitely many $n$.
        \item If $i\geq1$,  then $f_{DP}(i,i,n)\geq \frac{(2i+2)n}{i+2}$. This is sharp for each $i\geq 1$ for infinitely many $n$.
    \end{enumerate}
\end{theorem}

Note that depending on the relations between $i$ and $j$, we have five different (exact) bounds.
Since every non-$(i,j)$-colorable graph contains an $(i, j)$-critical subgraph, Theorem~\ref{main} yields the following.

\begin{corollary}\label{main_}
Let $G$ be a multigraph.
    \begin{enumerate}
        \item If $j\geq 1$ and for every subgraph $H$ of $G$, $|E(H)|\leq |V(H)|+j-1$, then $G$ is $(0,j)$-colorable. This is sharp.
        
        \item If $i\geq1$, $j\geq 2i+1$ and for every subgraph $H$ of $G$,  $|E(H)|\leq \frac{(2i+1)|V(H)|-(2i-j+2)}{i+1}$, then $G$ is $(i,j)$-colorable.
       This is sharp.
        \item If $i\geq1$, $i+2\leq j\leq 2i$ and for every subgraph $H$ of $G$, $|E(H)|\leq \frac{2j|V(H)|+1}{j+1}$, then $G$ is $(i,j)$-colorable.
        This is sharp.
        \item If $i\geq1$ and for every subgraph $H$ of $G$,  $|E(H)|\leq \frac{(2i^2+4i+1)}{i^2+3i+1}|V(H)|$, then $G$ is $(i,i+1)$-colorable.
        This is sharp. 
        \item If $i\geq1$ and for every subgraph $H$ of $G$, $|E(H)|\leq \frac{(2i+2)|V(G)|-1}{i+2}$, then $G$ is $(i,i)$-colorable. This is sharp.
    \end{enumerate}
\end{corollary}

Since a version of our construction in Section \ref{sec:construction} for $(0,j)$-colorings is a simple graph,  Part 1 of Corollary \ref{main_} implies the following result.

\begin{corollary}\label{simple:0j}
Let $G$ be a simple graph and $j\geq 1$ be integers. If for every subgraph $H$ of $G$, $|E(H)|\leq |V(H)|+j-1$, then $G$ is $(0,j)$-colorable. This is sharp 
for all $j\geq 1$ and each $n\geq 3j+3$.
\end{corollary}

In the next section we prove the lower bound in Part 5 of Theorem~\ref{main}. For other lower bounds we will use a more general framework.
It will be introduced in Section~\ref{Se4}, 
and in the subsequent five sections we prove the more general versions of the four other lower bounds. In the last section, we present constructions showing that our bounds are sharp for each $i\leq j$ for infinitely many $n$.

\section{Proof of the lower bound in Theorem \ref{main} for $(i,i)$-colorings}\label{sec:dd}

In this section, we prove 
 the lower bound in Part 5 of Theorem~\ref{main}. 
 The proof adjusts to DP-coloring the idea of Hendrey and Wood in \cite[Theorem 7]{HW18} for list coloring.

\begin{prop}\label{prop:dd}
Let $i\geq 1$ be an integer, and $G$ be an $(i,i)$-critical graph. Then $|E(G)|\geq \frac{2i+2}{i+2}|V(G)|$.
\end{prop}


\begin{lemma}\label{L11inHendrey} Let $k\geq 0$.
If  a cover $\Cov{H} = (L,\HH)$ 
 of a graph $G$ {satisfies}
 \begin{equation}\label{HW1}
\mbox{ for every $u\in V(G)$, $d(u) +1 \leq |L(u)| (i+1)$,}
\end{equation}
 then $\Cov{H}$ is $(i,i)$-colorable.
\end{lemma}

\begin{proof} Choose an $\Cov{H}$-map $\phi$ with minimum  $|E(\HH_{\phi})|$. Suppose there is $v\in V(G)$ such that $d_{\HH_{\phi}}(\phi(v))\geq i+1$. 
By~(\ref{HW1}), there is $\alpha\in L(v)-\phi(v)$ such that 
$$|N(\alpha)\cap V(\HH_{\phi})|\leq \left\lfloor\frac{d(v)}{|L(v)|}\right\rfloor\leq   \left\lfloor\frac{|L(v)| (i+1)-1}{|L(v)|}\right\rfloor=i.$$
  Define a map $\phi'$ as follows: $\phi'(v) = \alpha$, and $\phi'(u) = \phi(u)$ for every $u\in V(G)\setminus \{v\}$. Then $|E(\HH_{\phi'})|\leq |E(\HH_{\phi})| - 1$, a contradiction. 
\end{proof}
Let $G$ be an $(i,i)$-critical graph. 
  For $X\subset V(G)$ and $v\in V(G)$, let $d_X(v)=|N(v)\cap X|$.

\begin{lemma}\label{L13inHendrey} For   every 
 partition $V(G) = A\sqcup B$ with $A\neq\emptyset$ and $B\neq \emptyset$, there is $v\in B$ such that $(i+1)d_A(v) + d_B(v) \geq 2i+2$.
\end{lemma}
\begin{proof}  Suppose there is a partition $A\sqcup B$   with $A\neq\emptyset$ and $B\neq \emptyset$ such that
\begin{equation}\label{HW3}
\mbox{\em for every $v\in B$, $(i+1)d_A(v) + d_B(v) \leq 2i+1$.}
\end{equation}
 Let $\Cov{G} = (L,\HH)$ be a $2$-fold cover on $G$ such that $\HH$ does not have an 
$(i,i)$-coloring.
 Let $\HH^A$ (respectively,  $\HH^B$) denote the subgraph of $\HH$ corresponding to $G[A]$ (respectively, $G[B]$). Since $G$ is $(i,i)$-critical, $\HH^A$ has an 
 $(i,i)$-coloring $\phi$. For every $v\in B$, form $L'(v)$ from $L(v)$ by excluding from it every $v_{\alpha}$  such that $v_{\alpha}$ has a neighbor in $\HH^A_{\phi}$. 
 Then $|L'(v)|\geq 2 - d_A(v)$ for each $v\in B$. By~\eqref{HW3}, this is at least
 $2-\frac{2i+1-d_B(v)}{i+1}= \frac{d_B(v)+1}{i+1}$. Hence by Lemma~\ref{L11inHendrey}, $\HH^B$ has an $(i,i)$-coloring $\phi'$. Then the representative map $\psi$ defined by $\psi(w) = \phi(w)$ for $w\in A$ and $\psi(w) = \phi'(w)$ for $w\in B$ is an $(i,i)$-coloring on $\HH$, a contradiction to the choice of $G$. 
\end{proof}

Let $v_1,\dots, v_p\in V(G)$ be a maximal sequence such that for every $k\in[p]$, 
$$(i+1)d_{A_k}(v_k)+d_{B_k}(v_k)\geq 2i+2, \mbox{\em  where $A_k := \{v_1,\dots, v_k\}, B_k := V(G)\setminus A_k$. }$$
By Lemma~\ref{L13inHendrey}, $p = |V(G)|$. Then 
$$
\sum_{k = 1}^{|V(G)|}\Big((i+1)d_{A_k}(v_k)+d_{B_k}(v_k)\Big) \geq (2i+2)|V(G)|.
$$
On the other hand, every edge of $G$ contributes to the sum $\sum_{k = 1}^{|V(G)|}\Big((i+1)d_{A_k}(v_k)+d_{B_k}(v_k)\Big)$ exactly  $i+2$.
It follows that  $(i+2)|E(G)|\geq (2i+2)|V(G)|$, 
as claimed. This proves Proposition~\ref{prop:dd}.

\section{A more general model}\label{Se4}
When $j>i$, we will need the following more general
 framework. Instead of $(i,j)$-colorings of a cover $\Cov{H}$ of a graph $G$, we will consider $\Cov{H}$-maps $\phi$ with variable restrictions 
 on the degrees of the vertices in $H_\phi$. Furthermore, we will define potentials of vertex subsets of $G$ so that the lower is a potential of a set $W$, the larger is the average degree of $G[W]$. We will prove existence of our variable colorings in graphs with no subsets of ``low" potential, and will derive  our main result,
   Theorem~\ref{main}, as a partial case of our bounds.

For a graph $G$, a {\em toughness function on $G$} is a mapping ${\bf t} : V(G)\to \{0,1,\ldots,j+1\}$. A pair $(G,{\bf t})$ where $G$ is a graph and
${\bf t}$ is a toughness function will be called a {\em weighted pair}.

\begin{defn}[An $(i,j, {\bf t})$-coloring]\label{def3t} 
 {\em
Given a weighted pair $(G,{\bf t})$ and a cover $\Cov{H} = (L, \HH)$  of $G$, an {\em $(i,j, {\bf t})$-coloring} of $\HH$ is
a $\Cov{H}$-map $\phi$
		such that  the degree of every poor vertex $p(v)$ in $\HH_\phi$ is at most $i-{\bf t}(v)$ and  the degree of every rich vertex $r(v)$ in $\HH_\phi$ is at most 
		$j-{\bf t}(v)$. (If $i-{\bf t}(v)<0$ (respectively, $j-{\bf t}(v)<0$), this means $\phi(v)$ cannot be $p(v)$ (respectively,
		$\phi(v)$ cannot be $r(v)$).
		
  A vertex $v \in V(G)$ is $k$-{\em tough} in $(G,{\bf t})$ if ${\bf t}(v)=k$. }
\end{defn}

If ${\bf t} \equiv 0$, then any $(i,j, {\bf t})$-coloring of a graph $G$ is an  $(i,j)$-coloring in the  sense of Definition~\ref{def3}. So, Definition~\ref{def3t} is 
a refinement of Definition~\ref{def3}. Similarly the next definition refines Definition~\ref{ijcr}.

\begin{defn}[$(i,j)$-critical pairs.]\label{ijcr2}
Given $0\leq i\leq j$ and a weighted pair $(G,{\bf t})$, we say that  $(G,{\bf t})$ is   \emph{$(i, j)$-critical}, if $G$ is not $(i, j,{\bf t})$-colorable, 
but every proper subgraph of $G$ is. 
\end{defn}

We will measure the sparsity of our graphs with  so called  potential function.



\begin{defn}\label{DEF-P}
If $j\neq i+1$ or $j\leq 2$, given  a weighted pair $(G,{\bf t})$, the {\em $(i,j,{\bf t})$-potential} of a vertex $v\in V(G)$ is defined by 
\begin{equation}\label{eq:rho1}
\rho_{G,{\bf t}}(v):= a_{i,j}+{\bf t}(v)\cdot (a_{i,j}-2b_{i,j}),
\end{equation}
where $a_{i,j}:= b_{i,j}:=1$ when $i=0$, $a_{i,j}:=2i+1$ and $b_{i,j}:=i+1$ when $i \geq 1$ and $j \geq 2i+1$,  $a_{i,j}:=2j$ and $b_{i,j}:=j+1$ 
when $i\geq 1$ and $2i \geq j \geq i+2$. 

In other words,
\begin{equation}\label{eq:rho}
\rho_{G,{\bf t},i,j}(v):= \left\{\begin{array}{ll} 1-{\bf t}(v),& \mbox{ if } i=0;\\
2i+1-{\bf t}(v),& \mbox{ if } i\geq 1  \mbox{ and }j\geq 2i+1;\\
2j-2{\bf t}(v),& \mbox{ if } i\geq 1  \mbox{ and }i+2\leq j\leq 2i.
\end{array}
\right.
\end{equation}
For a subset $S\subseteq V(G)$, the $(i,j,{\bf t})$-potential of $S$ is defined by
\begin{equation}\label{eq:PS}
\rho_{G,{\bf t},i,j}(S):= \sum_{v\in S}\rho_{G,{\bf t}}(v) - b_{i,j}\cdot |E(G[S])|.
\end{equation}
 The $(i,j,{\bf t})$-potential of a graph $G$ is defined by $\rho_{{\bf t},i,j}(G):= \min_{S\subseteq V(G)}\rho_{G,{\bf t},i,j}(S)$.
\end{defn}
 
 When $i$ and $j$ are clear from the context, we will drop these subscripts from the notation $\rho_{G,{\bf t},i,j}(S)$ and will call 
 the $(i,j,{\bf t})$-potential of $S$ simply {\em the potential of $S$}.
Let $w_k(i,j)=a_{i,j}+k (a_{i,j}-2b_{i,j})$, i.e., $w_k(i,j)$ is the $(i,j,{\bf t})$-potential of a $k$-tough vertex in $(G,{\bf t})$. 
 
 In the next four sections  we prove the following theorem. 
\begin{theorem}\label{Potential form}
Let $(G,{\bf t})$ be an $(i, j)$-critical weighted pair, where  $j\neq i+1$ and ${\bf t}$ is an arbitrary toughness function. Then $\rho_{{\bf t},i,j}(G)\leq w_{j+1}(i,j)$.
In particular,

{\rm (1)} If $i=0$ and $j\geq 1$, then $\rho_{{\bf t},i,j}(G) \leq -j$.

{\rm (2)} If $i\geq 1$ and $j \geq 2i+1$, then $\rho_{{\bf t},i,j}(G) \leq 2i-j$.

{\rm (3)} If $i\geq 1$ and $2i \geq j \geq i+2$, then $\rho_{{\bf t},i,j}(G) \leq -2$.

\end{theorem}

Observe that if we take ${\bf t} \equiv 0$, then Theorem \ref{Potential form} implies the lower bounds of Parts  1, 2 and 3 of Theorem \ref{main}.
 In other words, we are proving a generalization of these parts of Theorem \ref{main}.

\section{Preliminaries}

For a graph $G$ and disjoint sets $U,W\subset V(G)$, $E_G(U,W)$ denotes the set of the edges of $G$ with one end in $U$ and one in $W$. If $U=\{u\}$ and $W=\{w\}$, then instead of $E_G(U,W)$ we write $E_G(u,w)$.
 If $e\in E_G(u,v)$ and 
 $\Cov{H} = (L,\HH)$ is a cover of $G$, then $M_\HH(e)$ (or simply $M(e)$ when $\HH$ is clear from the context) denotes the matching between $L(u)$ and
 $L(v)$ in $\HH$ corresponding to $e$.
 
 For $e\in E_G(u,v)$, a matching $M(e)$ is {\em even} if its edges are $r(u)r(v)$ and $p(u)p(v)$, and 
 is {\em odd} otherwise, i.e.,  if its edges are $r(u)p(v)$ and $p(u)r(v)$.


We will use the following lemmas at various points:

\begin{lemma}\label{lem:Jingweigeneral}
For nonnegative integers $i,j$ with $i\leq j$, suppose  Theorem \ref{Potential form} does not hold, and  $(G,{\bf t})$ is an
 $(i,j)$-critical pair  of minimum order with potential larger than $w_{j+1}(i,j)$.
 Then every nonempty $S\subsetneq V(G)$ with $\rho_{G,{\bf t}}(S)\leq w_{j}$ consists of a single $j$-tough vertex.
\end{lemma}
\begin{proof}
Suppose the lemma fails. Choose a maximum $S\subsetneq V(G)$ with $\rho_{G,{\bf t}}(S)\leq w_{j}$.
  Let $\Cov{H} = (L,\HH)$ be a cover on $G$ such that $\HH$ does not have an $(i,j)$-coloring. 
  Note that for every $v\in V(G)\setminus S$, we have $|N(v)\cap S|\leq 1$, since otherwise 
  $$\rho(S\cup\{v\})\leq w_{j}+(a_{i,j}-2b_{i,j}) = w_{j+1},$$ which contradicts the assumption on $(G,{\bf t})$. 
  Form a pair $(G',{\bf t'})$ from $(G,{\bf t})$ as follows: 
  \\
  (a) Let $V(G') = V(G)\setminus S \cup \{v^*\}$;
\\  
   (b) let ${\bf t'}(u)={\bf t}(u)$ for every $u\in V(G')\setminus\{v^*\}$ and ${\bf t'}(v^*) = j$; 
 \\  
   (c) for each edge $e\in E_G(uz)$ with $u\in S$ and $z\in V(G)\setminus S$, add an edge  between $v^*$ and $z$.
   
   If $\rho(G',{\bf t'})\leq w_{j+1}$, let $S'\subsetneq V(G')$ be a maximal subset with $\rho_{G',{\bf t'}}(S')\leq w_{{j}+1}$. By construction of $(G',{\bf t'})$, $v^*\in S'$. Let $S'' = S'\setminus\{v^*\}$. Then
$$\rho_{G,{\bf t}}(S''\cup S) = \rho_{G,{\bf t}}(S) + \rho_{G',{\bf t'}}(S') - \rho_{G',{\bf t'}}(v^*) \leq w_{j} + w_{j+1} - w_{j} = w_{j+1},$$ a contradiction to $\rho(G,{\bf t})>w_{j+1}$. Hence such $S'$ does not exist, and  $\rho(G',{\bf t'})>w_{j+1}(i,j)$.

Denote the subgraph of $\HH$ induced by $\HH[L(S)]$ by $\HH[S]$. Since $(G,{\bf t})$ is $(i,j)$-critical, $\HH[S]$ has an $(i,j,\t)$-coloring $\phi$. For every 
$z\in N_{G}(S)$ and its neighbor $u\in S$,
for each $e\in E_G(u,z)$, denote the neighbor of $\phi(u)$ in $M(e)$ by $z_a(e)$, and the other vertex in $L(z)$ by $z_b(e)$. Let $\Cov{G'}' = (L', \HH')$ be a cover of $G'$, such that : \\
1) $L'(v^*) = \{p(v^*),r(v^*)\}$; \\
2) for every $z\in N(v^*)$ and every edge $e\in E_G(S,z)$, $p(v^*)$ is adjacent to $z_b(e)$ and $r(v^*)$ is adjacent to $z_a(e)$; \\
3) for every edge $xy\in E(G')$ such that neither of $x$ nor $y$ is equal to $v^*$, $\HH'[\{x,y\}] = \HH[\{x,y\}]$. 

Then by the minimality of $(G,{\bf t})$, $\HH'$ has an $(i,j,\t)$-coloring $\psi$. Since ${\bf t'}(v^*) = j$, $\psi(v^*) = r(v^*)$ and $r(v^*)$ has degree $0$ in $\HH'_{\psi}$. Now we define an $\Cov{H}$-map  $\sigma$  by $\sigma(z)=\phi(z)$ for every $z\in S$, and $\sigma(z) = \psi(z)$ for every $z\in V(G)\setminus S$. By the construction of $G'$, for every $vu\in E(G)$ such that $v\in S$ and $u\in V(G)\setminus S$, $\sigma(v)$ is not adjacent to $\sigma(u)$. Hence $\sigma$ is an $(i,j,{\bf t})$-coloring of $\HH$, a contradiction.
\end{proof}

\begin{lemma}\label{degree} If $(G,{\bf t})$ is an $(i,j)$-critical pair and $v\in V(G)$ with $d(v)=1$,  then  $\t(v) \geq i+1$. 
\end{lemma}
\begin{proof}
Suppose $\t(v)\leq i$, $d(v)=1$ and $N(v)=\{u\}$. Given an arbitrary cover $\Cov{H}=(L,\HH)$ of $G$, we consider the  graph $\HH'=\HH-L(v)$.
Since $(G,{\bf t})$ is $(i,j)$-critical, $\HH'$ has an $(i,j,{\bf t})$-coloring
 $\phi$. We extend $\phi$ to $\HH$ by letting $\phi(v)$ be the vertex in $L(v)$  not adjacent to $\phi(u)$.  
\end{proof}

\section{Proof of Theorem \ref{Potential form} for $(0,j)$-colorings}\label{sec:0k}
In this section, we  prove Part 1 of Theorem \ref{Potential form}: 
\begin{prop}\label{prop:0k}
Let $j\geq1$ be an integer, and let $(G,{\bf t})$ be a $(0,j)$-critical pair. Then $\rho(G,{\bf t})\leq-j$.
\end{prop}
Recall that $a_{0,j}=b_{0,j}=1$. By (\ref{eq:rho}), $w_k=1-k$, for every $k\in \{0,1,\dots,j+1\}$. 
 Suppose  the proposition does not hold, and  $(G,{\bf t})$ is a
 $(0,j)$-critical pair   with potential larger than $w_{j+1}(0,j)=-j$  with the minimum $|V(G)|+|E(G)|$.
 Let $\Cov{H} = (L,\HH)$ be a cover of $G$ such that $\HH$ does not have a $(0,j,\t)$-coloring. 
First, we analyse the structure of $G$ and $\Cov{H}$.

\begin{lemma}\label{evenmatching}
For  every edge $e\in E(G)$,  matching  $M(e)$ is even.
\end{lemma}

\begin{proof}
Suppose there exists  $e\in E(G)$ such that  $M(e)$  is odd. For definiteness, suppose $e\in E_G(u,v)$.

{\bf Case 1:} There is  an $e'\in E_G(u,v)-e$ such that $M(e')$ is even. Let $G'=G-e-e'$ and define 
$\t'(x) = \t(x)$ for every $x\in V(G)-\{u,v\}$, and $\t'(x) = \t(x)+1$ when $x\in \{u,v\}$. We claim that
\begin{equation}\label{o27}
\rho(G',\t')\geq 1-j=w_j(0,j).
\end{equation}
Indeed, assume $\rho_{G',{\bf t'}}(S)\leq -j$. By the definition of ${\bf t'}$, $S\cap \{u,v\}\neq \emptyset$, say $u\in S$. If also $v\in S$, then
$|E(G[S])|-|E(G'[S])|=2$, and hence by the definition of potentials, $\rho_{G,\t}(S)=\rho_{G',\t'}(S)\leq -j$, a contradiction. Thus $v\notin S$ and hence
$\rho_{G,\t}(S)=\rho_{G',\t'}(S)+1\leq -j+1$.
Then by Lemma~\ref{lem:Jingweigeneral}, $S=\{u\}$ and $\t(u)=j$. But in this case, 
$$\rho_{G,\t}(\{u,v\})=\rho_{G,\t}(u)+\rho_{G,\t}(v)-|E_G(u,v)|\leq (1-j)+1-2=-j,$$
a contradiction. This proves~\eqref{o27}.

 Let $\HH'$ be the cover graph on $G'$ obtained from $\HH$ by deleting 
 $M(e)$ and $M(e')$. By the minimality of $(G,\t)$, $\HH'$ has a $(0,j,\t')$-coloring $\sigma$. For $x\in \{u,v\}$, since $\t'(x)\geq 1$, we have
 $\sigma(x) = r(x)$ 
 and $d_{\HH'_\sigma}(x)\leq j-\t(x)-1$. Since only one edge in $M_\HH(e)\cup M_\HH(e')$ connects $r(u)$ with $r(v)$, $d_{\HH_\sigma}(x)\leq j-\t(x)$, and hence
  $\sigma$ is  a $(0,j,\t)$-coloring on $\HH$, a contradiction.

{\bf Case 2:} For every $e'\in E_G(u,v)$, $M(e')$ is  odd.
Form $G''$ from $G$ by deleting all edges between $u$ and $v$ and gluing $u$ and $v$ into a new vertex $v^*$.
 Let $\t''(x) = \t(x)$ for every $x\in V(G'')- v^*$, and $\t''(v^*) = \max\{\t(u),\t(v)\}$. Since 
 $$\rho_{G'',\t''}(\{v^*\})=1-\max\{\t(u),\t(v)\}\geq (1-\t(u))+(1-\t(v))-|E_G(u,v)|= \rho_{G,\t}(\{u,v\}),$$
  we get $\rho(G'',\t'')\geq \rho(G,\t)\geq 1-j$. 
 Let $ \HH''$ be the cover graph on $G''$ obtained from $\HH$ by deleting the edges between $L(u)$ and $L(v)$ and by gluing $r(u)$ with $r(v)$ into the new vertex $r(v^*)$ and gluing $p(u)$ with $p(v)$ into the new vertex $p(v^*)$.
 By the minimality of $(G,\t)$, $\HH''$ has a $(0,j,\t'')$-coloring $\psi$. 
 Then the map $\phi$, where $\phi(x) = \psi(x)$ for every $x\in V(G)- u-v$ and $\phi(u) = \phi(v)=\psi(v^*)$, is a $(0,j,\t)$-coloring of $\HH$, a contradiction.
\end{proof}

\begin{lemma}\label{0j:neighbor}
For every $0$-tough $v\in V(G)$, $|N_G(v)|\geq 3$. 
\end{lemma}

\begin{proof}
   Suppose $|N_G(v)|\leq 2$ for some $0$-tough $v\in V(G)$.
       
{\bf Case 1:} $ |N_G(v)|=1$, say $N_G(v)=\{u\}$.  Since $(G,\t)$ is $(0,j)$-critical,
       $\HH-L(v)$ has a $(0,j,\t)$-coloring $\sigma$. 
       By Lemma~\ref{evenmatching}, for every $e\in E_G(v,u)$, matching $M(e)$ is even.
 Extend $\sigma$ to $v$ by choosing $\sigma(v)\in L(v)$ not adjacent to $\sigma(u)$.  Then $\sigma$ is a $(0,j,\t)$-coloring on $\HH$, a contradiction. 
    
{\bf Case 2:} $ |N_G(v)|=2$, say $N_G(v)=\{u,w\}$, and $u$ is not adjacent to $w$ in $G$. Let $G'$ be obtained from $G-v$ by gluing $u$ with $w$
 into the new vertex $u^*$. Let $\t'(x) = \t(x)$ for every $x\in V(G')- u^*$, and $\t'(u^*) = \max\{\t(u),\t(w)\}$. Since 
 $$\rho_{G',\t'}(\{u^*\})=1-\max\{\t(u),\t(w)\}\geq (1-\t(u))+(1-\t(w))+(1-\t(v))-2$$
 $$\geq (1-\t(u))+(1-\t(w))+(1-\t(v))-
 |E_G(v,\{u,w\})|= \rho_{G,\t}(\{v,u,w\}),$$
 we get $\rho(G',\t')\geq \rho(G,\t)\geq 1-j$. 
 Let $\HH'$ be the cover graph of $G'$ obtained from $\HH-L(v)$  by gluing $r(u)$ with $r(w)$ into the new vertex $r(u^*)$ and gluing $p(u)$ with $p(w)$ into the new vertex $p(u^*)$.
 By the  minimality of $G$, $\HH'$ has a $(0,j,\t')$-coloring $\psi$. 
 Define $\phi(x) = \psi(x)$ for every $x\in V(G)- u-w-v$,  $\phi(u) = \phi(w)=\psi(u^*)$, and choose $\phi(v)\in L(v)$ not adjacent to $\phi(u)$.
By Lemma~\ref{evenmatching}, $\phi$  is a $(0,j,\t)$-coloring of $\HH$, a contradiction.

{\bf Case 3:} $ |N_G(v)|=2$, say $N_G(v)=\{u,w\}$, and $u$ is  adjacent to $w$ in $G$, say $e\in E_G(u,w)$.  Let $G''$ be obtained from $G-v$ by adding an extra edge $e'$ connecting $u$ and $w$. Let $\t''(x) = \t(x)$ for every $x\in V(G'')$. Suppose $\rho_{G'',\t''}(S)\leq -j$ for some $S\subset V(G'')$. 
Since $\t''(x) = \t(x)$ for every $x\in V(G'')$, $u,w\in S$ and $\rho_{G,\t}(S)\leq \rho_{G',\t'}(S)+1\leq 1-j$. Then 
$$\rho_{G,\t}(S\cup \{v\})\leq \rho_{G,\t}(S)+\rho_{G,\t}(v)-|E_G(v,S)|\leq (1-j)+1-2=-j,$$
a contradiction. Thus
 $\rho(G'',\t'')\geq 1-j$. 
Let $ \HH''$ be the cover graph on $G''$ obtained from $\HH-L(v)$   by adding an odd matching connecting $L(u)$ and $L(w)$. 
 By the  minimality of $(G,\t)$, $\HH''$ has a $(0,j,\t'')$-coloring $\psi$. Since $\HH''$ has both, odd and even, matchings connecting $L(u)$ and $L(w)$,
 $\psi(u)=r(u)$ and $\psi(w)=r(w)$. Then by choosing $\psi(v)=p(v)$ we get a $(0,j)$-coloring on $\HH$, a contradiction.
\end{proof}

 If $d(v)\leq j-{\bf t}(v)$ for each $v\in V(G)$, then we color each $v\in V(G)$ with $r(v)$ and
 obtain a $(0,j,\t)$-coloring of $G$. Thus  there is a vertex $v_0\in V(G)$ such that 
 \begin{equation}\label{o272}
 d(v_0)\geq j+1-{\bf t}(v_0).
 \end{equation}

\medskip

By (\ref{eq:PS}), every edge $e$ contributes potential $-b_{0,j}=-1$ to the potential of a subset $S$ containing the ends of $e$. 
 We will view this as if each edge $e$ has charge  $ch(e)=-1$ and each vertex $v$ has charge $ch(v)=1-\t(v)$. By the choice of $G$,
 $\sum_{x\in V(G)\cup E(G)}ch(x)\geq 1-j$. We will use discharging to show that this is not the case.
The discharging rules are as follows.

 {\bf (R1)} Every edge incident to $v_0$ gives  charge  $-1$ to $v_0$.

 {\bf (R2)} Every edge not incident to $v_0$ gives  charge  $-1/2$ to each of its ends.
 


Denote the new charge of a vertex $v \in V(G)$ by $\mu(v)$. Note that after discharging, every edge has charge 0.
 So 
$$\rho(G,\t)= \sum_{v \in V(G)}\mu(v).$$

\medskip

By~\eqref{o272} and (R1),
$$\mu(v_0)=\rho_{G,\t}(v_0)-d(v_0) \leq (1-{\bf t}(v_0))-(j-{\bf t}(v_0)+1)=-j.$$
If $v\neq v_0$ is $0$-tough, then by
Lemma~\ref{0j:neighbor}, it has at least two neighbors distinct from $v_0$.  So by (R2),
$$\mu(v) \leq \rho_{G,\t}(v)-2 \times 1/2 \leq 1-1=0.$$
Finally, if $v$ is  $k$-tough for some $k\geq 1$, then 
$\mu(v)=\rho_{G,\t}(v)=1-k\leq 0.$
Therefore $$\rho(G,\t)= \sum_{v \in V(G)}\mu(v) \leq \mu(v_0)=-j,$$
 a contradiction. This proves the proposition.


\section{Proof of Theorem \ref{Potential form} for  $i\geq1$ and $j\geq 2i+1$}\label{sec:ijlarge}

In this section we prove Part 2 of Theorem \ref{Potential form}:
\begin{prop}\label{prop:ijlarge}
Let $i\geq1$ and $j\geq 2i+1$ be integers, and let $(G,\t)$ be an $(i,j)$-critical pair. Then $\rho(G,\t)\leq2i-j$.
\end{prop}

Recall that  in this case, $a_{i,j}=2i+1$ and $b_{i,j}=i+1$. By (\ref{eq:rho}), $w_k=2i+1-k$ for every $k\in \{0,1,\dots,j+1\}$. 
 Suppose  the proposition does not hold, and  $(G,{\bf t})$ is a
 $(i,j)$-critical pair   with potential larger than $w_{j+1}(i,j)=2i-j$  with the minimum $|V(G)|+|E(G)|$.
 Let $\Cov{H} = (L,\HH)$ be a cover of $G$ such that $\HH$ does not have an $(i,j,\t)$-coloring. 


\begin{lemma}\label{LM-J-1}
$G$ contains at most one $j$-tough vertex.
\end{lemma}
\begin{proof}
If $u$ and $v$ are two $j$-tough vertices, 
then, since $2i+1-j\leq 0$,
 $$\rho_{G,\t}(\{u,v\})\leq \rho_{G,\t}(u)+\rho_{G,\t}(v)=2w_j(i,j)=2(2i+1-j)\leq 2i+1-j,$$ 
contradicting Lemma~\ref{lem:Jingweigeneral}. 
\end{proof}

\begin{lemma}\label{LM2}
Every edge in $G$ is incident to a $j$-tough vertex.
\end{lemma}
\begin{proof}
Suppose there is $e\in E_G(x,y)$ such that neither $x$ nor $y$ is $j$-tough. Let  $G'=G - e$ and  define $ \t'(w) = \t(w)$ for every $w\in V(G')\setminus\{x,y\}$ and $\t'(w) = \t(w)+1$ for  $w\in \{x,y\}$. Let $\HH'$ be formed from $\HH$ by deleting $M(e)$. Then $\Cov{H}' = (L,\HH')$ is a cover on $G'$.

Suppose there is $S\subseteq V(G')$ such that $\rho_{G', \t'}(S)\leq 2i-j$. 
Then $S\cap \{x,y\}\neq \varnothing$, say $x\in S$. If also $y\in S$, then  $|E(G[S])|=1+|E(G'[S])|$ and hence
$$\rho_{G,\t}(S) \leq\rho_{G',\t'}(S)+2-(i+1) \leq 2i-j+2-(i+1)\leq 2i-j,$$
a contradiction. So let  $y \notin S$.  Then $\rho_{G,\t}(S) \leq \rho_{G',\t'}(S) + 1 \leq 2i-j+1$. So by Lemma~\ref{lem:Jingweigeneral},
 $S=\{x\}$ and $x$ is $j$-tough,
   a contradiction. Hence  $\rho(G',\t') \geq 2i-j+1$. By the minimality of $G$, $\HH'$ has an $(i,j,\t')$-coloring $\phi$. Since adding $M(e)$ back to
   $\HH'$ may increase in $\HH'_\phi$ only the degrees of $\phi(x)$ and $\phi(y)$ and only by at most $1$, $\phi$ is also an 
   $(i,j,\t)$-coloring on $\HH$, a contradiction.
\end{proof}

Lemmas~\ref{LM-J-1} and~\ref{LM2} together imply that $G$ has
 a $j$-tough vertex $v_0$ such that each edge of $G$ is incident with $v_0$. If for some $v\in V(G)-v_0$, $|E_G(v,v_0)|\geq 2$, then
  $$\rho_{G,\t}(\{v,v_0\}) \leq w_j+(2i+1)-2(i+1)=2i-j=w_{j+1}.$$ So $G$ has no multiple edges, and $G$ is a star.
  
Again, let $v\in V(G)-v_0$.  
  By Lemma~\ref{degree}, $\t(v)\geq i+1$.  Hence 
  $$\rho_{G,\t}(\{v,v_0\}) \leq w_{i+1}+w_j-(i+1)=(2i+1-i-1)+(2i+1-j)-(i+1)=2i-j,$$ a contradiction.  This  proves Proposition~\ref{prop:ijlarge}.

\section{Proof of Theorem \ref{Potential form} for  $i\geq1$ and $i+2\leq j\leq 2i$}\label{sec:ijsmall}

In this section we prove Part 3 of Theorem \ref{Potential form}:
\begin{prop}\label{prop:ijsmall}
Let $i\geq1$ and  $i+2\leq j\leq 2i$ be integers, and let $(G,\t)$ be an $(i,j)$-critical pair. Then $\rho(G,\t)\leq-2$.
\end{prop}

The proof is very similar to the proof of Proposition~\ref{prop:ijlarge}. 
In this case, $a_{i,j}=2j$, $b_{i,j}=j+1$, and $w_k=2j-2k$ for all $k$. In particular, $w_j=0$. 
Suppose  the proposition does not hold, and  $(G,{\bf t})$ is a
 $(i,j)$-critical pair   with potential larger than $w_{j+1}(i,j)=-2$  with the minimum $|V(G)|+|E(G)|$.
Let $\Cov{H} = (L,\HH)$ be a cover of $G$ such that $\HH$ does not have an $(i,j,\t)$-coloring.

Since $w_j=0$, the following lemmas have the same statements and practically the same simple proofs as Lemmas~\ref{LM-J-1} and~\ref{LM2} (so, we omit the proofs).

\begin{lemma}\label{LM-Ma-1}
$G$ contains at most one $j$-tough vertex.
\end{lemma}


\begin{lemma}\label{LM-Ma-2} Every edge in $G$ is incident to a $j$-tough vertex.
\end{lemma}
 

As in Section~\ref{sec:ijlarge},
Lemmas~\ref{LM-Ma-1} and \ref{LM-Ma-2} together imply that $G$ has
 a $j$-tough vertex $v_0$ such that each edge of $G$ is incident with $v_0$. If for some $v\in V(G)-v_0$, $|E_G(v,v_0)|\geq 2$, then
  $\rho_{G,\t}(\{v,v_0\}) \leq w_j+2j-2(j+1)i=-2=w_{j+1}$. So $G$ has no multiple edges, hence is a star.
  
Let $v\in V(G)-v_0$.  
  By Lemma~\ref{degree}, $\t(v)\geq i+1$.  Hence 
  $$\rho_{G,\t}(\{v,v_0\}) \leq w_{i+1}+w_j-(j+1)=(2j-2(i+1))+0-(j+1)=j-2i-3.$$ 
 Since $j-2i\leq 0$, this is a contradiction   proving
 Proposition~\ref{prop:ijsmall}.

\section{Proof of the lower bound in Theorem \ref{main} for  $i\geq1$ and $j=i+1$}\label{sec:iiplus1}

In this section, we introduce a more flexible toughness function, and will use it to prove a generalization of the lower bound in Part 4 of Theorem \ref{main}.

\subsection{A more refined model}\label{ssec:ii+1}

We modify the definitions in Section 4 as follows.

For a graph $G$, a {\em toughness function on $G$} is a function ${\bf t}$ mapping each $v\in V(G)$ into a pair $({\bf t}_p(v),{\bf t}_r(v))$, where
${\bf t}_p(v)\in \{0,1,\ldots,i+1\}$ and ${\bf t}_r(v)\in \{0,1,\ldots,j+1\}$.
 A pair $(G,{\bf t})$ where $G$ is a graph and
${\bf t}$ is a toughness function will be called a {\em weighted pair}.

\begin{defn}[An $(i,j, {\bf t})$-coloring]\label{def3t2} 
 {\em
Given a weighted pair $(G,{\bf t})$, and a cover $\Cov{H} = (L, \HH)$  of $G$, an {\em $(i,j, {\bf t})$-coloring} of $\HH$ is
a $\Cov{H}$-map $\phi$
		such that  the degree of every poor vertex $p(v)$ in $\HH_\phi$ is at most $i-{\bf t}_p(v)$ and  the degree of every rich vertex $r(v)$ in $\HH_\phi$ is at most 
		$j-{\bf t}_r(v)$.
		
  A vertex $v \in V(G)$ is $(k_1,k_2)$-{\em tough} in $(G,{\bf t})$ if ${\bf t}(v)=(k_1,k_2)$. }
\end{defn}

\begin{defn}\label{Potential C} 
Given a weighted pair $(G,{\bf t})$ and its cover $\Cov{H}=(L, \HH)$, the potential of a vertex $v \in V(G)$ is defined by 
\begin{equation}\label{112}
\rho_{G,{\bf t}}(v):=\left\{
\begin{aligned}
   &2i^2+4i+1-(i+1){\bf t}_p(v)-i{\bf t}_r(v) &\text{ if }{\bf t}_p(v)-{\bf t}_r(v)\geq0,\\
   &2i^2+4i+1-i{\bf t}_p(v)-(i+1){\bf t}_r(v) &\text{ if }{\bf t}_p(v)-{\bf t}_r(v)<0.
\end{aligned}
\right.
\end{equation}
For a subset $S \subset V(G)$, the potential of $S$ is defined by 
$$\rho_{G,{\bf t}}(S):=\sum_{v \in S}\rho_{G,{\bf t}}(v)-(i^2+3i+1)|E(G[S])|.$$
The potential of $(G,{\bf t})$ is defined by $\rho(G,{\bf t}):=\min_{S\subset V(G)}\rho_{G,{\bf t}}(S)$
\end{defn}

The definition of critical pairs is the same as in Section 4.


The main result of the section is: 
\begin{theorem}\label{prop:ii+1}
Let $i \geq 1$ be an integer. If $(G,{\bf t})$ is $(i,i+1)$-critical, then $\rho(G,{\bf t}) \leq -1.$     
\end{theorem}
Observe that if we take ${\bf t}(v)=(0, 0)$ for every $v \in G$, then Theorem \ref{prop:ii+1} yields the lower bound in Part 4 of Theorem \ref{main}.


\subsection{Lemmas for the proof of Theorem~\ref{prop:ii+1}}\label{ssec:ii+12}
Suppose the theorem does not hold. Then we can choose an $(i,i+1)$-critical pair
 $(G,{\bf t})$ with $\rho(G,{\bf t}) \geq 0$ that has 
 minimum possible $|V(G)|+|E(G)|$ and modulo this --- the maximum $\rho(G,{\bf t})$.
 
 We start from a useful observation.
 
 \begin{lemma}\label{spl}
Pair  $(G,{\bf t})$
has no vertices $v$ with $\t_r(v)\geq i+2$. 
\end{lemma}

\begin{proof} Suppose $v\in V(G)$ and $\t_r(v)\geq i+2$. Since $\rho_{G,\t}(v)\geq 0$, $\t_p(v)\leq i-1$.
For  every $2$-fold cover $\HH_1$ of $(G,{\bf t})$, in each $(i,i+1,{\bf t})$-coloring $\phi$ of $\HH_1$,
\begin{equation}\label{1121}
\phi(v)=p(v) \;\mbox{ \em and
 $|\bigcup_{u\in V(G)-v}\{e\in E_G(v,u): \; p(v)\phi(u)\in M_{\HH_1}(e)\}|\leq i-\t_p(v)$. }
\end{equation}
Let ${\bf t'}(v)=(i+1,\t_p(v))$ and ${\bf t'}(w)={\bf t}(w)$ for each $w\in V(G)-v$.  
For  each $2$-fold cover $\HH_1$ of $(G,{\bf t})$, 
let $\HH'_1$ be obtained from $\HH_1$ by switching the parities of 
the matchings $M_{\HH_1}(e)$ for all edges $e$ incident with $v$. By~\eqref{1121} and the definitions of ${\bf t'}$ and $\HH'_1$,
any subgraph of $\HH_1$ (including the whole $\HH_1$) has an $(i,i+1,{\bf t})$-coloring if and only if the corresponding subgraph of
$\HH'_1$ has an $(i,i+1,{\bf t'})$-coloring. Thus, since $(G,{\bf t})$ is $(i,i+1)$-critical, the pair $(G,{\bf t'})$ also is $(i,i+1)$-critical.
But  by~\eqref{112}, 
$$\rho_{G,{\bf t'}}(v)-\rho_{G,{\bf t}}(v)=2i^2+4i+1-(i+1)(i+1)-i\t_p(v)-(2i^2+4i+1-i\t_p(v)-(i+1)\t_r(v))$$
$$=(i+1)(\t_r(v)-(i+1))>0,$$
 and hence $\rho(G,{\bf t'})>\rho(G,{\bf t})$, a contradiction to the choice of
$(G,{\bf t})$.
 \end{proof}
 
We now derive an analog of Lemma~\ref{lem:Jingweigeneral} (with almost the same proof):

\begin{lemma}\label{cor:uniqueii+1}
For every $S\subsetneq V(G)$, if $\rho_{G,{\bf t}}(S) \leq i$, then $S$ is  a  $(i+1,i+1)$-tough vertex. In particular,  $\rho_{G,{\bf t}}(S)=i$.
\end{lemma}

\begin{proof} Let $S$ be a largest proper subset of $V(G)$ with 
  $\rho_{G,{\bf t}}(S)\leq i$. If $S=\{v\}$, then to have $0\leq \rho_{G,{\bf t}}(S)\leq i$, by~\eqref{112}, $v$ is either $(i+1,i+1)$-tough or  $(i,i+2)$-tough.
  But the latter is excluded by Lemma~\ref{spl}, thus our lemma holds in this case. So, suppose $|S|\geq 2$.
  
 If there is $u\in \overline{S}$ such that $|E_G(S,u)|\geq 2$, then $$\rho_{G,{\bf t}}(S\cup \{u\})\leq i + 2i^2+4i+1 - 2(i^2+3i+1) = -i-1,$$ a contradiction. Thus,
 \begin{equation}\label{1122}
\mbox{ \em 
for each $u\in \overline{S}$, \;  $|E_G(S,u)|\leq 1$. }
\end{equation}
 Let $\Cov{H} = (L,\HH)$ be a cover on $G$ such that $\HH$ does not have an $(i,j)$-coloring. 
   Form a pair $(G',{\bf t'})$ from $(G,{\bf t})$ as follows: 
  \\
  (a) Let $V(G') = V(G)\setminus S \cup \{v^*\}$;
\\  
   (b) let ${\bf t'}(w)={\bf t}(w)$ for every $w\in V(G')\setminus\{v^*\}$, and let ${\bf t'}(v^*) = (i+1,i+1)$; 
 \\  
   (c) for each edge $e\in E_G(uw)$ with $u\in S$ and $w\in V(G)\setminus S$, add an edge   between $v^*$ and $w$.

\smallskip   
   If $\rho(G',{\bf t'})\leq -1$, let $S'\subsetneq V(G')$ be a maximal subset with $\rho_{G',{\bf t'}}(S')\leq -1$. By construction of $(G',{\bf t'})$, $v^*\in S'$. Let $S'' = S'\setminus\{v^*\}$. Then
$$\rho_{G,{\bf t}}(S''\cup S) = \rho_{G,{\bf t}}(S) + \rho_{G',{\bf t'}}(S') - \rho_{G',{\bf t'}}(v^*) \leq i + (-1) - i= -1,$$ a contradiction to $\rho(G,{\bf t})>-1$. 
This yields
  $\rho(G',{\bf t'})\geq 0$.

 Since $(G,{\bf t})$ is $(i,i+1)$-critical, $\HH[S]$ has an $(i,j)$-coloring $\phi$. 
 
For every 
$z\in N_{G}(S)$ and its neighbor $u\in S$,
for each $e\in E_G(u,z)$, denote the neighbor of $\phi(u)$ in $M(e)$ by $z_a(e)$, and the other vertex in $L(z)$ by $z_b(e)$. Let $\Cov{H}' = (L', \HH')$ be a cover of $G'$, such that : \\
1) $L'(v^*) = \{p(v^*),r(v^*)\}$; \\
2) for every $z\in N(v^*)$ and every edge $e\in E_G(S,z)$, $p(v^*)$ is adjacent to $z_b(e)$ and $r(v^*)$ is adjacent to $z_a(e)$; \\
3) for every edge $xy\in E(G')$ such that neither  $x$ nor $y$ is equal to $v^*$, $\HH'[\{x,y\}] = \HH[\{x,y\}]$. 

Then by the minimality of $(G,{\bf t})$ and the fact that  $\rho(G',{\bf t'})\geq 0$, $\HH'$ has an $(i,i+1,\t)$-coloring $\psi$.
Since  ${\bf t'}(v^*) = (i+1,i+1)$,  $\psi(v^*) = r(v^*)$ and $r(v^*)$ has degree $0$ in $\HH'$. Now we define an $\Cov{H}$-map  $\sigma$  by $\sigma(z)=\phi(z)$ for every $z\in S$, and $\sigma(z) = \psi(z)$ for every $z\in V(G)\setminus S$. By the construction of $G'$, for every $vu\in E(G)$ such that $v\in S$ and $u\in V(G)\setminus S$, $\sigma(v)$ is not adjacent to $\sigma(u)$. Hence $\sigma$ is an $(i,i+1,{\bf t})$-coloring of $\HH$, a contradiction.
 \end{proof}

\subsection{Low sets and vertices}
The following notion is quite useful. A {\em low set} is a proper subset  $S$ of $V(G)$ with  $\rho_{G,{\bf t}}(S) \leq 2i$.

\begin{lemma}\label{LM-M-2} 
If $S$  is a low set with $|S|\geq 2$ and $|E[S, V(G)\setminus S]| \geq 2$,
 then $|S|=2$, $S$ is independent, and each $x\in S$ is  $(i+1,i+1)$-tough.
\end{lemma}
\begin{proof} Suppose lemma is not true. Choose a largest low set $S$ with $|S|\geq 2$ and $|E[S, V(G)\setminus S]| \geq 2$
that   is not an independent set of two vertices. 

Let $E_G(S,V(G)-S)=\{e_1,\ldots,e_k\}$ where $e_h$ connects $v_h\in S$ with $u_h\in V(G)-S$ for $h=1,\ldots,k$ ( some vertices can coincide).
Under the conditions of the lemma, $k\geq 2$.
Construct $G'$  by adding  to $G-S$ two new vertices $x$ and $y$ and the set of edges $\{e'_1,\ldots,e'_k\}$ where $e_h$ connects $x$ with $u_h$ for
$h=1,\ldots,k-1$ and $e_k$ connects $y$ with $u_k$.

We claim that 
\begin{equation}\label{1124}
|V(G')|+|E(G')|< |V(G)|+|E(G)|. 
\end{equation}
Indeed, $|V(G')|\leq |V(G)|$ since $|S|\geq 2$, and $|E(G')|\leq |E(G)|$ by construction. Moreover, if
we have equalities in both inequalities, then $|S|=2$ and $|E(G[S])|=0$, a contradiction to the choice of $S$.
This proves~\eqref{1124}.

Let ${\bf t'}(u) = (i+1,i+1)$ for $u\in \{x,y\}$ and ${\bf t'}(u) = {\bf t}(u)$ for $u\in V(G)\setminus S$.
 
 If there is $S'\subseteq V(G')$ such that $\rho_{G',{\bf t'}}(S')<0$, then $S'\cap \{x,y\}\neq \emptyset$, say $x\in S'$. If also $y\in S'$, then
 $$\rho_{G,{\bf t}}((S'-\{x,y\})\cup S)\leq (-1)-\rho_{G,{\bf t}}(\{x,y\})+\rho_{G,{\bf t}}(S)\leq -1,$$
 a contradiction. On the other hand, if $x\in S'$ and $y\notin S'$, then
 $\rho_{G,{\bf t}}(G[S'\setminus\{x\}\cup S])<0-i+2i = i$.  Since $| S|\geq 2$, this  contradicts Lemma~\ref{cor:uniqueii+1}. 
 Hence $\rho(G',{\bf t'})\geq 0$. Denote the subgraph of $\HH$ induced by $G[S]$ by $\HH_S$.
 Since $(G,\t)$ is $(i,i+1)$-critical, $\HH_S$ has an $(i,i+1,\t)$-coloring $\phi_1$.

For every $1\leq h\leq k$, denote the neighbor of $\phi_1(v_h)$ in $M(e_h)$ by $a(u_h)$, and the other vertex in $L(u_h)$ by $b(u_h)$.
 Let $\Cov{H}' = (L', \HH')$ be the cover of $G'$, such that : \\
1) for every $1\leq h\leq k-1$, $p(x)$ is adjacent to $b(u_h)$ and $r(x)$ is adjacent to $a(u_h)$; \\
2)  $p(y)$ is adjacent to $b(u_k)$ and $r(y)$ is adjacent to $a(u_k)$; \\
3) $\HH'[L'(V(G')\setminus \{x,y\})] = \HH[L(V(G)\setminus S)]$.

 By~\eqref{1124} and the fact that $\rho(G',{\bf t'})\geq 0$, $\HH'$ has an $(i,i+1,\t)$-coloring $\phi_2$. 
Since  ${\bf t'}(x) = {\bf t'}(y) =(i+1,i+1)$, 
\begin{equation}\label{1125}
 \mbox{\em $\phi_2(x) = r(x)$,  $\phi_2(y) = r(y)$, and each of   $r(x)$ and $r(y)$ has degree $0$ in $\HH'$. }
 \end{equation}
 Define a representative map $\phi$ on $\HH$ by letting $\phi(v) = \phi_1(v)$ for every $v\in S$ and $\phi(v) = \phi_2(v)$ for every $v\in V(G)\setminus S$. 
By~\eqref{1125}, for each $1\leq h\leq k$,   $\phi(v_h)$ is not adjacent to $\phi(u_h)$.
 Thus $\phi$ is an $(i,i+1,{\bf t})$-coloring of $\HH$, a contradiction.
\end{proof}

Similarly to a low set, a {\em low vertex} is a vertex $v$ with $\rho_{G,\t}(v)\leq 2i$.

\begin{lemma}\label{LM-M-22} 
Every low set consists of either one low vertex or two  $(i+1,i+1)$-tough vertices.
\end{lemma}
\begin{proof} Suppose there exists a low set $S$ with $|S|\geq 3$.
By Lemma~\ref{LM-M-2}, $|E_G(S,V(G)-S)|=1$. So we may assume $E_G(S,V(G)-S)=E_G(x,y)=\{e\}$ where
$x\in S$ and $y\notin S$.

Let $G'=G-S$ and $\t'$ be defined by $\t'(y)=(\t_p(y)+1,\t_r(y)+1)$ and $\t'(w)=\t(w)$ for all $w\in V(G)-S-y$.
By the definition of $\t'$, if $\rho(G',\t')\leq -1$, then there exists a low set $S'\subset V(G)-S$ with $y\in S'$.
But in this case,  
$\rho_{G,\t}(S\cup S')\leq 2i+2i-(i^2+3i+1)=i-i^2-1\leq -1$, a contradiction to $\rho(G,\t)\geq 0$. Thus
$\rho(G',\t')\geq 0$, and by the minimality of $G$, $G'$
has  an $(i,i+1,{\bf t'})$-coloring $\phi$. Let $a\in \{p,r\}$ be such that $a(x)$ is the neighbor of $\phi(y)$ in $\HH$.

{\bf Case 1:} Vertex $x$ is not $(i+1,i+1)$-tough. Let $\t''$ differ from $\t$ on $S$ only in that $\t''_a(x)=\t_a(x)+1$.
 By Lemma~\ref{cor:uniqueii+1}, $\rho(G[S],\t'')\geq 0$. So by the minimality of $G$, $G[S]$
has  an $(i,i+1,{\bf t''})$-coloring $\psi$. We claim that $\phi \cup \psi$ is an $(i,i+1,{\bf t})$-coloring of $G$.
 Indeed,  if $\psi(x)\neq a(x)$, this is trivial, and if $\psi(x)=a(x)$, this follows from the definitions of $\t'$ and $\t''$.

{\bf Case 2:} Vertex $x$ is  $(i+1,i+1)$-tough. Let $G_3=G[S]$ and $G_4=G-(S-x)$. By the minimality of $G$,
$G_3$ has  an $(i,i+1,{\bf t})$-coloring $\phi_3$ and $G_4$ has  an $(i,i+1,{\bf t})$-coloring $\phi_4$. Since 
$x$ is  $(i+1,i+1)$-tough, $\phi_3(x)=\phi_4(x)=r(x)$, and $r(x)$ has neighbors neither in $H_{\phi_3}$ nor in $H_{\phi_4}$.
But then $\phi_3\cup \phi_4$ is an $(i,i+1,{\bf t})$-coloring of $G$, a contradiction.
\end{proof}

\begin{lemma}\label{danger} 
For every $v\in V(G)$, at most one edge connects $v$ with a low vertex.
\end{lemma}
\begin{proof} Suppose  for $h\in [2]$, $e_h\in E_G(v,u_h)$ and $u_h$ is low (possibly, $u_1=u_2$). Then
$$\rho_{G,\t}(\{v,u_1,u_2\})\leq 2(2i)+(2i^2+4i+1)-2(i^2+3i+1)=2i-1.$$
Then
  by Lemma~\ref{LM-M-22},  $V(G)=\{v,u_1,u_2\}$. Furthermore, in order to have $\rho({G,\t})\geq 0$, we need $E(G)=\{e_1,e_2\}$ and
  $\max\{\rho_{G,\t}(u_1),\rho_{G,\t}(u_2)\}\}>i$, say, $\rho_{G,\t}(u_1)>i$. Then either ${\bf t}_p(u_1)\leq i-1$ or ${\bf t}_r(u_1)\leq i$,
  say ${\bf t}_r(u_1)\leq i$. In this case, we let $\phi(u_1)=r(u_1)$, let $\phi(u_2)$ be any  color $\alpha\in\{p(u_2),r(u_2)\}$ of nonnegative capacity,
  and choose $\phi(v)\in \{p(v),r(v)\}$ not adjacent to $\alpha$. By construction, the only possibility that $\phi$ is not
  an $(i,i+1,{\bf t})$-coloring of $G$ is that $\phi(v)r(u_1)\in E(\HH)$ and either $\phi(v)=p(v)$ and $\t_p(v)\geq i$ or 
 $\phi(v)=r(v)$ and $\t_r(v)\geq i+1$. Since $\rho_{G,\t}(v)\leq 2i^2+4i+1-(i+1)\max\{\t_p(v),\t_r(v)\}$, in order to have $\rho({G,\t})\geq 0$, 
 we need $\max\{\t_p(v),\t_r(v)\}\leq 1$, which yields $i=1$ and so $\rho_{G,\t}(v)\leq 2(1^2)+4(1)+1-(1+1)1=5$.
 Hence $\rho_{G,\t}(\{v,u_1,u_2\})\leq 2(2\times 1)+5-2(1^2+3(1)+1)=9-2(5)=-1$, a contradiction.
\end{proof}

\subsection{Potentials of the vertices of small degree}
\begin{lemma}\label{CLM-M-1} 
If $v \in V(G)$ is a leaf, then $\rho_{G,{\bf t}}(v) \leq 2i$.
\end{lemma}
\begin{proof} 
Suppose $\rho_{G,{\bf t}}(v) \geq 2i+1$, and  $N(v)=\{u\}$. Then either ${\bf t}_p(v)=i+1$ or ${\bf t}_r(v)=i+2$, since otherwise we can  extend to $v$
 any $(i,i+1,{\bf t})$-coloring  of $G-v$. Moreover, by Lemma~\ref{spl}, the latter cannot hold. Thus ${\bf t}_p(v)=i+1$ and
 \begin{equation}\label{113}
 \rho_{G,{\bf t}}(v) \leq 2i^2+4i+1-(i+1)(i+1)
  = i^2+2i.
   \end{equation}
 On the other hand, since $\rho_{G,{\bf t}}(v) \geq 2i+1$,
 \begin{equation}\label{1132}
 \mbox{\em  ${\bf t}_r(v)\leq i$. }
 \end{equation}
 Let $\beta(u)\in \{p(u),r(u)\}$ be the neighbor of $p(v)$ in $\HH$ and
 $\overline{\beta}(u)\in L(u)-\beta(u)$. 
Let $G'=G-v$ and let $\t'$ differ from $\t$ on $V(G')$ only for $\t'(u)$, where
the toughness of $\overline{\beta}(u)$ increases by $1$. Since the potential of each subset of $V(G')$ decreases by at most $i+1$, the only possibility that
$\rho({G',\t'})\leq -1$ is that $\rho_{G,\t}(u)\leq i$. But in this case by~\eqref{113},  $\rho_{G,\t}(\{u,v\})\leq i+(i^2+2i)-(i^2+3i+1)=-1$, a contradiction. Thus
$\rho({G',\t'})\geq 0$ and hence by the minimality of $G$, $G'$ has  an $(i,i+1,{\bf t'})$-coloring $\phi$. 
Extend $\phi$ to $v$ by letting $\phi(v)=r(v)$. If
$\phi(u)=\beta(u)$, then we do not create conflicts, and if  $\phi(u)=\overline{\beta}(u)$, then $\phi$ is   an $(i,i+1,{\bf t})$-coloring of $G$ because of~\eqref{1132} and
the definition of $\t'$.
\end{proof}

\begin{lemma}\label{d2} If $v\in V(G)$ and
 $d_G(v)=2$, then $\rho_{G,{\bf t}}(v) \leq i^2+3i+1$.
\end{lemma}
\begin{proof}  
Suppose  
 $\rho_{G,{\bf t}}(v) \geq i^2+3i+2$ and $E_G(v,V(G)-v)=\{e_1,e_2\}$ where $e_h=vu_h$ for  $h\in [2]$.  
 Then by~\eqref{112}, $\t_p(v)\leq i-1$ and  $\t_r(v)\leq i-1$. By Lemma~\ref{danger}, we may assume $u_1$ is not low.
 Let $G'=G-v$ and let $\t'$ differ from $\t$ on $V(G')$ only in that $\t'(u_1)=(\t_p(u_1)+1,\t_r(u_1)+1)$.
 We claim that $\rho(G',\t')\geq 0$. Indeed, suppose $\rho_{G',\t'}(S)\leq -1$ for some $S\subseteq V(G')$. By the definition of $\t'$ this 
 implies that $u_1\in S$ and $\rho_{G,\t}(S)\leq (2i+1)+\rho_{G',\t'}(S)\leq 2i$. Since $u_1$ is not low, this contradicts  Lemma~\ref{LM-M-22}.
  
 Thus $\rho(G',\t')\geq 0$, and by the minimality of $G$, $G'$ has an $(i,i+1,{\bf t'})$-coloring $\phi$. Extend $\phi$ to $v$ by letting $\phi(v)$ be the color
 $\alpha\in L(v)$ not adjacent to $\phi(u_2)$. If $\alpha$ is not adjacent to $\phi(u_1)$, then $d_{\HH_\phi}(\alpha)=0$, but even if $\alpha\phi(u_1)\in E(\HH)$,
 then by the choice of $\t'$ and the fact that $\t_p(v)\leq i-1$ and  $\t_r(v)\leq i-1$, $\phi$ is  an $(i,i+1,{\bf t})$-coloring of $G$. 
 \end{proof}

\begin{lemma}\label{CLM-M-2} If $v\in V(G)$ and $d_G(v)=3$, then $\rho_{G,{\bf t}}(v) \leq i^2+4i+2$.
\end{lemma}
\begin{proof} 

Suppose  
 $\rho_{G,{\bf t}}(v) \geq i^2+4i+3$ and $E_G(v,V(G)-v)=\{e_1,e_2,e_3\}$ where $e_h=vu_h$ for $ h\in [3]$ (some $u_h$ can coincide).
 By Lemma~\ref{danger}, we may assume that $u_1$ and $u_2$ are not low.
  Let $G'=G-v$. Define $\t'(x)=\t(x)$ for all $x\in V(G')-\{u_1,u_2\}$
 and $\t'(u_h)=(\t_p(u_h)+1,\t_r(u_h)+1)$ for $h\in [2]$ (if $u_1=u_2$, then $\t'(u_1)=(\t_p(u_1)+2,\t_r(u_1)+2)$).
 
 Suppose there is $S\subseteq V(G')$ with $\rho_{G',\t'}(S)\leq -1$. Since $u_1$ and $u_2$ are not low,
by Lemma~\ref{LM-M-22},  $\{u_1,u_2\}\subseteq S$, and 
 $\rho_{G,\t}(S)=2(2i+1)+\rho_{G',\t'}(S)\leq 4i+1$. Then 
 \begin{equation}\label{114}
\rho_{G,\t}(S\cup \{v\})\leq (4i+1)+\rho_{G,\t}(v)-2(i^2+3i+1)\leq 2i-1.
 \end{equation}
 Since $v$ is not low,  Lemma~\ref{LM-M-22} and~\eqref{114} yield that $S\cup \{v\}=V(G)$.
 But if $S\cup \{v\}=V(G)$, then in~\eqref{114} we did not take $e_3$ into account. So, instead of~\eqref{114}, we have
$$\rho_{G,\t}(S\cup \{v\})\leq (4i+1)+\rho_{G,\t}(v)-3(i^2+3i+1)\leq 
-i^2-i-1, 
 $$
 a contradiction.
\end{proof}

\subsection{Discharging}
 At the start, each vertex $v$ has charge $ch(v)=\rho_{G,{\bf t}}(v)$ and
each edge $e$  has charge $ch(e)=-(i^2+3i+1)$. Then
$$ 
\sum_{x\in V(G)\cup E(G)}ch(x)=\sum_{v\in V(G)}\rho_{G,\t}(v)-(i^2+3i+1)|E(G)|=\rho_{G,{\bf t}}(V(G))\geq 0.$$

In the discharging, every edge $e$ gives charge $-0.5(i^2+3i+1)$ to each of its ends. Denoting the resulting charge 
of an $x\in V(G)\cup E(G)$ by $\mu(x)$, we obtain that $\mu(e)=0$ for each $e\in E(G)$, so that
 \begin{equation}\label{charge}
\sum_{v\in V(G)}\mu(v)=\sum_{x\in V(G)\cup E(G)}\mu(x)=\sum_{x\in V(G)\cup E(G)}ch(x)=\rho_{G,{\bf t}}(V(G))\geq 0.
 \end{equation}
On the other hand, for each $v\in V(G)$, $\mu(v)=\rho_{G,{\bf t}}(v)-d(v)\frac{i^2+3i+1}{2}$, and so

\begin{itemize}
    \item If $d(v) \geq 4$, then $\mu(v) \leq (2i^2+4i+1)-2(i^2+3i+1) < 0$;
    \item If $d(v) = 3$, then by Lemma \ref{CLM-M-2}, $\rho_{G,{\bf t}}(v) \leq i^2+4i+2$, then $\mu(v) \leq (i^2+4i+2)-3(i^2+3i+1)/2 < 0$;
    \item If $d(v)=2$, then by Lemma \ref{d2}, $\mu(v) \leq (i^2+3i+1)-(i^2+3i+1)=0$;
    \item If $d(v)=1$,  then by Lemma \ref{CLM-M-1}, $\mu(v) \leq 2i-(i^2+3i+1)/2 < 0$.
\end{itemize}

By~\eqref{charge}, this implies that $\sum_{v\in V(G)}\mu(v)=\rho(G,{\bf t})\leq 0$
with equality only if each vertex has degree $2$ and potential exactly $i^2+3i+1$. Since $G$ is connected, it must be a cycle.
Furthermore, if $v\in V(G)$ with potential $i^2+3i+1$ has $\t_p(v)\leq i-1$ and  $\t_r(v)\leq i$, then the proof of Lemma~\ref{d2} still goes 
through. Therefore, for each $v\in V(G)$, $\t(v)=(i,0)$. But then we color every $v\in V(G)$ with $r(v)$. Since $i+1\geq 2$, this is 
 an $(i,i+1,{\bf t})$-coloring of $G$. 
This finishes the proof of Theorem~\ref{prop:ii+1}.

\section{ Constructions}\label{sec:construction}
Given a multigraph $G$, for every $v\in V(G)$ we say $v$ is a \emph{$d$-vertex} if $d(v)=d$. 
\begin{defn}[flags]
Given a vertex $v$ in a multigraph $G$, a {\em flag at $v$} is a $2$-cycle $vuv$ such that  $u$ has degree $2$, i.e., there are two edges
connecting $v$ with $u$, and no other edges incident to $u$. See Figure~\ref{fig:flag}, $v$ is the \emph{base vertex} of the flags, and $u_1,\dots,u_k$ are \emph{flag vertices}.
\end{defn}

\begin{defn}[weak flags]
Given a vertex $v$ in a multigraph $G$,  a \emph{weak flag of weight $i$ at $v$} is  subgraph of $G$ with vertex set 
$\{u_1,\ldots,u_{i},x,y\}$ such that $u_1,\ldots,u_{i}$ are flags at $x$ and $y$ is a $2$-vertex adjacent to $x$ and $v$;
 see Figure \ref{fig:weak}.
\end{defn}

\begin{figure}[h]
\centering
\begin{minipage}[t]{0.46\textwidth}
\centering
\includegraphics[width=2.6in]{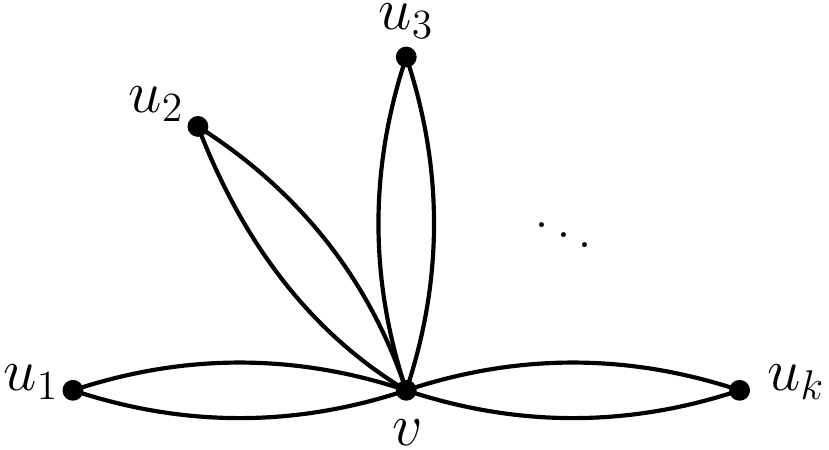}
\captionsetup{justification=centering}
\caption{A vertex $v$ with $k$ flags.}\label{fig:flag}
\end{minipage}\hfill\begin{minipage}[t]{0.54\textwidth}
\centering
\includegraphics[width=1.85in]{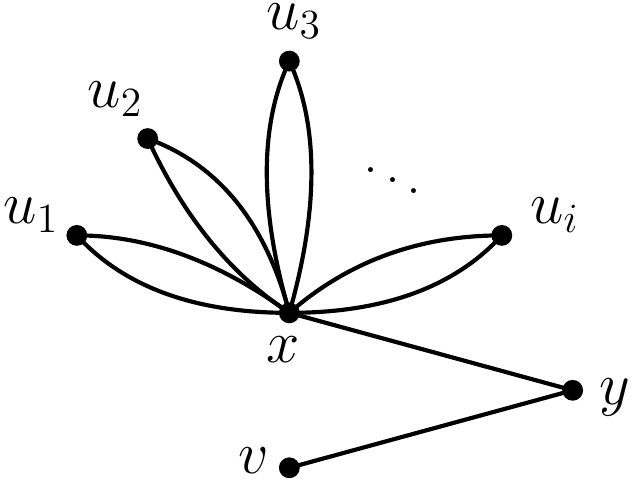}
\captionsetup{justification=centering}
\caption{A vertex $v$ with one weak flag of weight $i$.}\label{fig:weak}
\end{minipage}
\end{figure}



Call a vertex  $v\in V(G)$ a {\em $d^-$-vertex}  if $d_G(v)\leq d$, and a {\em $d^+$-vertex} in $G$ if $d_G(v)\geq d$.

We will use the following simple properties of $(i,j)$-critical graphs.

\begin{lemma}\label{quasi1} Let $0\leq i\leq j$ and $G$ be an $(i,j)$-critical graph. If $v$
 is a  vertex of $G$ with $N(v)=\{u\}$, and 
 $\Cov{H}=(L,\HH)$ is a cover of $G$ that does not have any $(i,j)$-coloring,
then some matching in $L(v)L(u)$ is even and some matching is odd. In particular, $\delta(G)\geq 2$, and for each flag vertex
$v$ with the neighbor $u$, one matching in $L(v)L(u)$ is even and one is odd.
\end{lemma}

\begin{proof} Suppose all the  matchings in $L(v)L(u)$ have the same parity. Let $G'=G-v$ and $\HH'=\HH-L(v)$. By the minimality of $G$, 
$\HH'$ admits an $(i,j)$-coloring $\phi$. Then we can choose $\phi(v)\in L(v)$ not adjacent to $\phi(u)$, a contradiction.
\end{proof}

\begin{lemma}\label{quasi} Let $0\leq i\leq j-1$ and $G$ be an $(i,j)$-critical graph. If $v$ is a vertex of $G$ with at most one  edge connecting $v$ with
a $3^+$-vertex, then $|N_G(v)|\geq i+2$.
\end{lemma}

\begin{proof} Suppose $N_G(v)=\{u_0,\ldots,u_s\}$, where  $s\leq i$ and all vertices apart from $u_0$ are $2^-$-vertices. Then by Lemma~\ref{quasi1},
$d_G(u_1)=\ldots=d_G(u_s)=2$.
If all $u_0,\ldots,u_s$ are flag vertices, then $V(G)=\{u_0,\ldots,u_s,v\}$. In this case,
we let $\phi(v)=r(v)$, and for $0\leq h\leq s$, choose $\phi(u_h)=r(u_h)$. By Lemma~\ref{quasi1},
  $r(v)$ has  in $H_\phi$ at most $s+1\leq i+1\leq j$ neighbors, and each $\phi(u_h)$ has at most one neighbor, contradicting the choice of $G$. 
  Thus, not all $u_0,\ldots,u_s$ are flag vertices, and we may assume that $u_0$ is not a flag vertex. Then under the conditions of the lemma,
 \begin{equation}\label{npr}
   \mbox{\em only one edge connects $v$ with $u_0$.
}
     \end{equation} 
     
 Let $\Cov{H}=(L,\HH)$ be a cover of $G$ that does not have any $(i,j)$-coloring.
Since $G$ is $(i,j)$-critical, the cover $\Cov{H}_1=\Cov{H}-L(v)$ of $G_1=G-v$ has an $(i,j)$-coloring $\phi$. 
By~(\ref{npr}), we can choose $\phi(v)\in L(v)$ not adjacent to $\phi(u_0)$.
Then for each $1\leq h\leq s$, we do the following: If $d_{H_\phi}(\phi(u_h))\leq 1$, then leave the color unchanged, and if
$d_{H_\phi}(\phi(u_h))=2$, then
 recolor $u_h$ with the other color, and
the degree of the new color will be $0$. This way, $d_{H_\phi}(\phi(v))\leq s\leq i$ and
$d_{H_\phi}(\phi(u_h))\leq 1\leq i$ for each $1\leq h\leq s$. 
Thus we obtained an $(i,j)$-coloring of $\Cov{H}$, a contradiction to its choice.
\end{proof}

\subsection{ Examples of $(0,j)$-critical graphs.}
We construct  $G_m$ as follows. Start from the cycle $v_0v_1\ldots v_mv_0$, and then for $h=1,\ldots,j$, add a 3-cycle $x_hy_hu_h$ where $u_h$ is adjacent to $v_0$,
see Figure~\ref{fig:0k}. 
By construction, for every integer $m$, $|V(G_m)|=3j+m+1$ and
  $|E(G_m)|=|V(G_m)|+j$. 

\begin{figure}[h]
\centering
\begin{minipage}[t]{0.46\textwidth}
\centering
\includegraphics[width=2.5in]{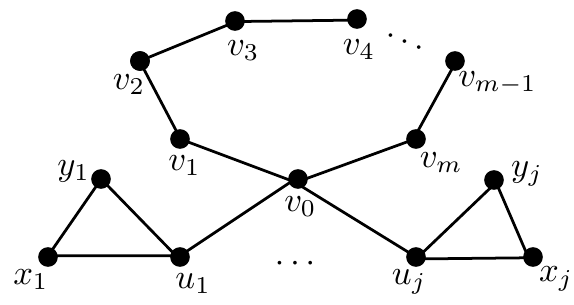}
\captionsetup{justification=centering}
\caption{Critical graphs for $(0,j)$-coloring.}\label{fig:0k}
\end{minipage}\hfill\begin{minipage}[t]{0.54\textwidth}
\centering
\includegraphics[width=3.4in]{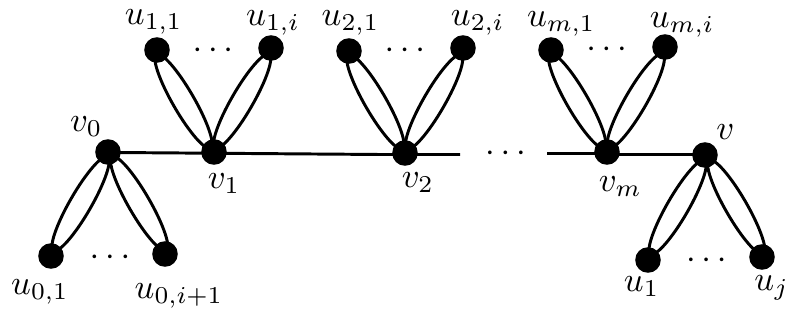}
\captionsetup{justification=centering}
\caption{Critical graphs for $(i,j)$-colorings when $j$ is large.}\label{fig:ijlarge}
\end{minipage}
\end{figure}

\begin{prop}
    $G_m$ is $(0,j)$-critical for every $m$.
\end{prop}
\begin{proof}
    First, we  construct a cover $\Cov{H}=(L,\HH)$ of $G_m$, such that there is no $(0,j)$-coloring of $\HH$.
 For all $h=1,\ldots,j$, we let the matchings  $L(u_i)L(x_i)$ and $L(u_i)L(y_i)$ be odd, and the matching $L(x_i)L(y_i)$ be even. 
  The matchings  $L(v_g)L(v_{g+1})$  for $0\leq g\leq m-1$ are odd, all remaining matchings in $\HH$ are even.
  
Suppose
 $\HH$ admits a $(0,j)$-coloring $\phi$. If for some $h\in [j]$, $\phi(u_h)=p({u_h})$, then since $L(u_i)L(x_i)$ and $L(u_i)L(y_i)$ are odd,
 we also have $\phi(x_h)=p(x_h)$ and $\phi(y_h)=p(y_h)$. But $L(x_i)L(y_i)$ is even, a contradiction. Thus, $\phi(u_h)=r(u_h)$ for all $h\in [j]$, and 
  if $\phi(v_0)=r(v_0)$, then (since $L(v_0)L(v_{m})$ is even) $\phi(v_m)=p(v_m)$. Since $L(v_{m-1})L(v_{m})$ is odd, this yields 
 $\phi(v_{m-1})=p(v_{m-1})$. Similarly, we get $\phi(v_{m-2})=p(v_{m-2})$, and so on. Finally, $\phi(v_{1})=p(v_{1})$, which means that $\phi(v_0)$ has
 $j+1$ neighbors in $\HH_\phi$, a contradiction. Hence we may assume $\phi(v_0)=p(v_0)$. Then symmetrically, $\phi(v_m)=r(v_m)$ and consecutively
 for $g=m-1,m-2,\ldots,1$ we obtain $\phi(v_{g})=r(v_{g})$. This means that  $\phi(v_0)$ has
  neighbor  $\phi(v_1)$ in $\HH_\phi$, a contradiction. So, $G_m$ is not $(0,j)$-colorable.
 
 Now, let $G'$ be any proper subgraph of $G$. Since every block of $G$ is a cycle or an edge, $G'$ has fewer than $j+1$ cycles or is disconnected.
 In both cases, $|E(G')|-|V(G')|\leq j-1$. So by Proposition~\ref{prop:0k}, no proper subgraph of $G$ is $(0,j)$-critical. Hence each proper subgraph of $G$
 is $(0,j)$-colorable. 
\end{proof}

\subsection{Examples  of $(i,j)$-critical graphs
for  $i\geq1$ and $j\geq2i+1$.}

Let $G_m$ be obtained from the path $v_0v_1\ldots v_mv$ by adding $j$ flags with base $v$, $i+1$ flags with base $v_0$ and for $h=1,\ldots,m$, $i$ flags with base $v_h$,  see Fig.~\ref{fig:ijlarge}. By construction, for every  $m\geq 0$, we have $|V(G_m)|=(i+1)(m+1)+2+j$ and
$|E(G_m)|=(2i+1)(m+1)+2+2j= \frac{(2i+1)|V(G_m)|+j-2i}{i+1}$.

\begin{prop}
    Suppose $i,j\geq1$ and $j\geq 2i+1$ be integers. Then $G_m$ is $(i,j)$-critical for every $m$.
\end{prop}
\begin{proof}
We construct a cover  $\Cov{H}=(L,\HH)$  of $G_m$  as follows. To each digon in $G_m$ correspond  one even and one odd matching in $\HH$. 
To each edge in the path  $v_0v_1\ldots v_m$ corresponds  an odd matching in $\HH$, and the matching corresponding to $vv_m$ is even.

Suppose $\HH$ admits an $(i,j)$-coloring $\phi$. Since $\phi(v_0)$ is adjacent to the $i+1$  vertices in $\{\phi(u_{0,1}),\ldots,\phi(u_{0,i+1})\}$ (see Fig.~\ref{fig:ijlarge}),
$\phi(v_0)=r(v_0)$. Since the matching corresponding to $v_0v_1$ is odd and $\phi(v_1)$ is adjacent to the $i$  vertices in 
$\{\phi(u_{1,1}),\ldots,\phi(u_{1,i})\}$, we also have $\phi(v_1)=r(v_1)$. Similarly, for each $2\leq h\leq m$ we conclude that $\phi(v_h)=r(v_h)$. 
Finally, $\phi(v)$ is adjacent to the $j$  vertices in $\{\phi(u_{1}),\ldots,\phi(u_{j})\}$. Hence $\phi(v)=r(v)$,
 and $\phi(v_m)$ cannot be adjacent to $\phi(v)$. But $\phi(v_m)=r(v_m)$ and
  the matching corresponding to $vv_m$ is even, a contradiction.

Thus if $G_m$ is not $(i,j)$-critical, then it contains a proper  $(i,j)$-critical subgraph $G'$. Suppose the cover $\Cov{H}'=(L',\HH')$  of $G'$
has no $(i,j)$-coloring. Let $q=|V(G_m)|-|V(G')|$.
By Proposition~\ref{prop:ijlarge}, $q\geq 1$. Also, $\delta(G')\geq 2$.
 Try to color $G'$ as follows: for each $w\in V(G')\cap \{v_0,\ldots,v_m,v\}$,
let $\psi(w)=r(w)$, and then for every remaining vertex $u$ (which is a flag vertex), choose $\psi(u)\in L'(u)$ with at most one edge connecting
$\psi(u)$ with $\psi(w)$ where $w$ is the neighbor of $u$ in $G'$. Since $\Cov{H}'$
has no $(i,j)$-coloring, some vertex $H'_\psi$  has more than $j$ neighbors. By the definition of $G$,
the only such vertex is
$\psi(v)$. Thus $v\in V(G')$ and also 
\begin{equation}\label{j29}
     \{u_1,\ldots,u_j,v_m\}\subset V(G')
     \end{equation} 
If not every $v_h$ belongs to $V(G')$, then let $g$ be the largest index such that $v_g\notin V(G')$. By~(\ref{j29}), $g\leq m-1$. Then
$v_{g+1}$  has at most one incident non-flag edge and
 is the base of at most $i$ flags, contradicting Lemma~\ref{quasi}. Thus $\{u_1,\ldots,u_j,v,v_m,v_{m-1},\ldots,v_0\}\subseteq V(G')$.
Hence, $G'$ is obtained from $G$ by deleting $q$ flag vertices and maybe some edges. But then, since $q\geq 1$,
$$|E(G')|\leq |E(G_m)|-2q= \frac{(2i+1)(|V(G')|+q)+j-2i}{i+1}-2q< \frac{(2i+1)|V(G')|+j-2i}{i+1},
$$
 contradicting Proposition~\ref{prop:ijlarge}.
\end{proof}

\subsection{Examples of $(i,j)$-critical graphs
for    $i+2\leq j\leq 2i$.}
Let $m\geq1$. Let $G_m$ be obtained from the path $P=v_0v_1\ldots v_{2m}$ by adding $j-1$ flags at each of $v_2,v_4,\ldots,v_{2m-2}$ and
adding $j$ flags at $v_0$ and $v_{2m}$.
Then $|V(G_m)|=(j+1)m+2+j$ and
 $|E(G_m)|= 2jm+2j+2=    \frac{2j|V(G_m)|+2}{j+1}$.


\begin{prop}
    Suppose  $i+2\leq j\leq 2i$. Then $G_m$ is $(i,j)$-critical for every $m\geq 1$.
\end{prop}
\begin{proof}
Since $i+2\leq 2i$, we have $i\geq 2$ and $j\geq 4$.
    We construct $\HH$ as follows. To each  digon in $G_m$ correspond two disjoint matchings in $\HH$. For each $1\leq h\leq m$, the
     matching between $L(v_{2h-2})$ and $L(v_{2h-1})$ is even, and the     matching between $L(v_{2h-1})$ and $L(v_{2h})$ is      odd. 
    
    Suppose $\HH$ admits an $(i,j)$-coloring $\phi$. Then for $h=0,\ldots,m$, $\phi(v_{2h})$ has at least $j-1\geq i+1$ neighbors in $H_\phi$,  hence
    $\phi(v_{2h})=r(v_{2h})$. Furthermore, $\phi(v_0)$ has $j$ neighbors in the flags based at $v_0$, so $\phi(v_1)$ is not adjacent to $\phi(v_0)$.
  Since the matching corresponding to $v_0v_1$ is even, this means $\phi(v_1)=p(v_1)$. 
 Since the matching corresponding to $v_1v_2$ is odd, it follows that $\phi(v_2)=r(v_2)$  is adjacent to $\phi(v_1)$ and to  $j-1$ neighbors
  in the flags based at $v_2$. So, similarly to the situation with $v_0$, $\phi(v_3)$ is not adjacent to $\phi(v_2)$ and is adjacent to $\phi(v_4)$. Repeating the argument, we get
  that $\phi(v_5)$ is  adjacent to $\phi(v_6)$, and so on. Finally, we get that $\phi(v_{2m-1})$ is  adjacent to $\phi(v_{2m})$, and hence $\phi(v_{2m})$ has
  $j+1$ neighbors in $H_\phi$, a contradiction.

  Thus, if $G_m$ is not $(i,j)$-critical, then it contains a  
     proper $(i,j)$-critical subgraph $G'$. Suppose the cover $\Cov{H}'=(L',\HH')$  of $G'$
has no $(i,j)$-coloring. Choose the smallest $k$ such that $v_k\in V(G')$ and suppose $k>0$. Since $G'$ is $(i,j)$-critical,
$\HH'-L(v_k)$ has an $(i,j)$-coloring $\phi$. Since $k>0$, $v_k$ has at most $j$ neighbors in $G'$ and hence 
       by Lemma~\ref{quasi1}, we can extend $\phi$ to $v_k$ as follows: If $k$ is even, then we let $\phi(v_k)=r(v_k)$, and if $k$ is
       odd, then we choose $\phi(v_k)\in L(v_k)$ not adjacent to $\phi(v_{k+1})$ (if $v_{k+1}$ is not in $V(G')$, then no restrictions).
 In both cases, we get an $(i,j)$-coloring of $\HH'$, a contradiction.
 
 Thus, $v_0\in V(G')$. Symmetrically, $v_{2m}\in V(G')$. Since every $(i,j)$-critical multigraph is connected, this means $\{v_0,\ldots,v_{2m}\}\subseteq V(G')$.
If $G'$ has exactly $s$ flag vertices, then $|V(G')|=2m+1+s$ and $|E(G')|\leq 2m+2s$. So  by Proposition~\ref{prop:ijsmall}, $s\geq 2+(m+1)(j-1)$, which means
$V(G')=V(G)$. Since the minimum degree of each $(i,j)$-critical multigraph is at least $2$, this yields $G'=G$, a contradiction.
   \end{proof}

\subsection{Examples of $(i,i+1)$-critical graphs
 for $i\geq1$.}
Let $v_0,\dots,v_{m+1}$ be the vertices of $P_{m+2}$, where $v_0$ and $v_{m+1}$ are end vertices. Define $G_m$ by adding $i+1$ weak flags of weight $i+1$ to $v_0$, adding $i$ weak flags of weight $i+1$ to $v_1,\dots,v_m$, and by adding $i+1$ flags to $v_{m+1}$, see Fig.~\ref{fig:ii+1}. 
Then for every $m$, $|V(G_m)|=(m+1)i^2+(3m+4)i+i+m+6$, $E(G_m)=2(m+1)i^2+4(m+2)i+m+7$, and thus
 $$|E(G_m)|= \frac{(2i^2+4i+1)|V(G_m)|+1}{i^2+3i+1}.$$

\begin{figure}[h]
    \centering
    \includegraphics[width=3.2in]{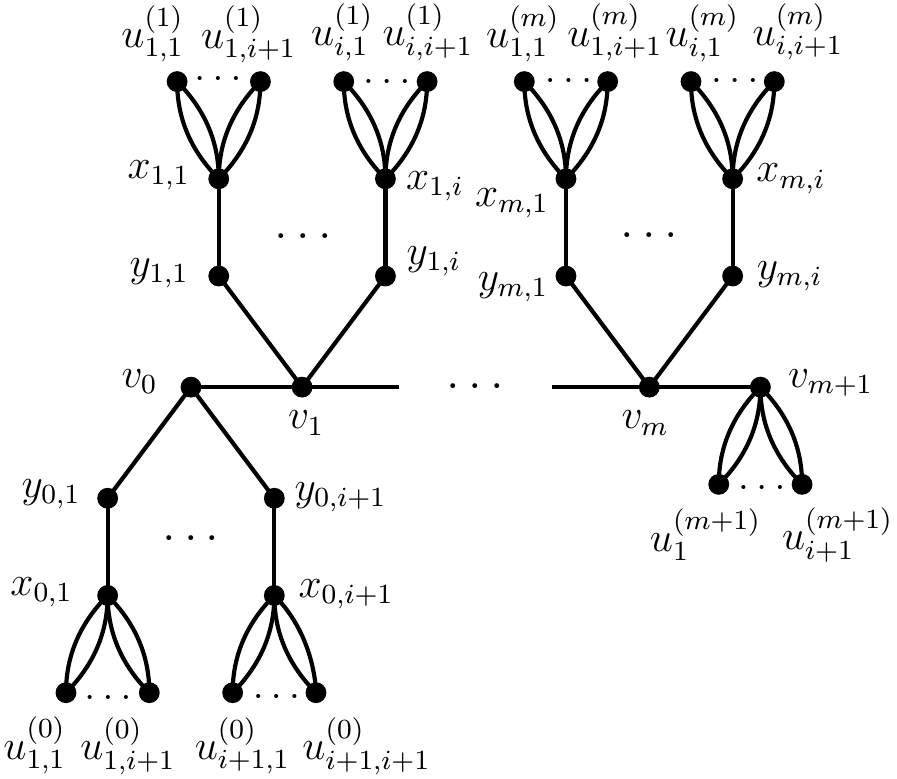}
\caption{Critical graphs for $(i,i+1)$-colorings.}\label{fig:ii+1}
\end{figure}

\begin{prop}
    Let $i\geq2$ be an integer. Then $G_m$ is $(i,i+1)$-critical for every $m$.
\end{prop}
\begin{proof}
    Let $\Cov{H}=(L,\HH)$ be a cover of $G_m$, we define $\HH$ as follows. Each digon in $G_m$ represents two disjoint matchings in $\HH$, and the edges in each weak flag but not in any digons represent even matchings. The matchings between $L(v_t)$ and $L(v_{t+1})$ are odd for every $t\leq m-1$, and the matching between $L(v_m)$ and $L(v_{m+1})$ is even. 
    
    Suppose $\HH$ admits an $(i,i+1)$-coloring $\phi$. Then $\phi(x_{\alpha,\beta})$, $\phi(v_{m+1})$ (See Fig.~\ref{fig:ii+1}) are rich for every possible $(\alpha,\beta)$, and thus $\phi(y_{\alpha,\beta})$ are poor for all $(\alpha,\beta)$. Hence $\phi(v_i)$ are rich for all $0\leq i\leq m$. But then $\phi(v_{m+1})$ has $i+2$ neighbors, a contradiction.

    If $G_m$ is not $(i,i+1)$-critical, then it contains a proper critical subgraph $G'$. Suppose the cover $\Cov{H}'=(L',\HH')$ of $G'$ has no $(i,i+1)$-coloring. By Proposition~\ref{prop:ii+1}, $|V(G_m)|-|V(G')|\geq1$, and $\delta(G')\geq2$. By Lemma~\ref{quasi}, $v_t\in V(G')$ for every $t=0,\dots,m+1$, and $u_s^{(m+1)}\in V(G')$ for $s=1,\dots,i+1$.
    We try to color $G'$ as follows: 
    
    Suppose one of the vertices of a weak flag of $v_t$ is not in $G'$, where $0\leq t\leq m$. We may assume $y_{t,1}\notin V(G')$ (By Lemma~\ref{quasi1} if $y_{t,1}\in V(G')$ then $x_{t,1}\in V(G')$, and by Lemma~\ref{quasi} if $x_{t,1}\in V(G')$ then all the flag vertices of $x_{t,1}$ are in $V(G')$). For every $w\in V(G')\cap(\bigcup\{x_{\alpha,\beta}\}\cup\{v_{m+1}\})$, let $\Phi(w)=r(w)$. Let $\Phi(y_{a,b})\in L'(y_{a,b})$ such that it is not adjacent to $\Phi(x_{a,b})$. For every $w\in\{v_0,\dots,v_{t-1}\}$, let $\Phi(w)=r(w)$. For every $t\leq s\leq m$, let $\Phi(v_s)\in L'(v_s)$ such that it is not adjacent to $\Phi(v_{s+1})$. Since $v_{t}$ has only at most $i+1$ neighbors in $G'$, $\Phi$ is an $(i,i+1)$-coloring of $\Cov{H}'$ of $G'$, a contradiction.
\end{proof}

\subsection{Examples of $(i,i)$-critical graphs
 for $i\geq1$.}

Let $G_m$ be obtained from the $2m$-cycle $C=v_0v_1\ldots v_{2m-1}v_0$ by adding $i$ flags (with flag vertices $u_{2h,1},\ldots,u_{2h,i}$) at vertex $v_{2h}$ 
for each $0\leq h\leq m-1$.
Then $|V(G_m)|=(i+2)m$ and 
  $|E(G_m)|=2m +2im=\frac{(2i+2)|V(G_m)|}{i+2}$ for every positive integer $m$.

\begin{prop}
    Let $i\geq1$ be an integer. Then $G_m$ is $(i,i)$-critical for every positive integer $m$.
\end{prop}

\begin{proof}
    Let $\Cov{H}_m=(L,\HH_m)$ be a cover of $G_m$, such that 
  \begin{equation}\label{m0}  
\mbox{ \em    the matching between $L(v_{2m-1})$ and $L(v_0)$ is odd}
\end{equation} and
 \begin{equation}\label{m01}  
\mbox{ \em
     for each $0\leq h\leq 2m-2$, 
    the matching between $L(v_{h})$ and $L(v_{h+1})$ is even.}
 \end{equation}   
     Also, one of the matchings between $L(v_{2h})$ and $L(u_{2h,q})$ is odd and the other is even
  for  each $0\leq h\leq m-1$  and each $1\leq q\leq i$. Suppose $\HH_m$ has an $(i,i)$-coloring $\phi$.
  Then for each $0\leq h\leq m-1$, vertex $\phi(v_{2h})$ has $i$ neighbors in the set $\{\phi(u_{2h,1}),\ldots,\phi(u_{2h,i})\}$. Therefore,
  $\phi(v_{2h})\nsim\phi(v_{2h-1})$ and  $\phi(v_{2h})\nsim\phi(v_{2h+1})$. By~\eqref{m01}, this yields that the parity of $\phi(v_{2h+1})$
  differs from the parities of $\phi(v_{2h})$ and $\phi(v_{2h+2})$ for each $0\leq h\leq m-1$. It follows that the parities of 
  $\phi(v_{0}),\phi(v_{2}),\ldots,\phi(v_{2(m-1)})$ are the same, and the parity of $\phi(v_{2m-1})$ is different from them. But this contradicts~\eqref{m0}. 
  
  
   Thus, if $G_m$ is not $(i,i)$-critical, then it contains a  
     proper $(i,i)$-critical subgraph $G'$. Suppose the cover $\Cov{H}'=(L',\HH')$  of $G'$
has no $(i,i)$-coloring. If $G'$ does not contain a vertex $v_{2h}$ for some $0\leq h\leq m-1$, then by Lemma~\ref{quasi1}, also $v_{2h-1}$ and
$v_{2h+1}$ are not in $G'$. But then by Lemma~\ref{quasi}, also $v_{2h-2}$ and
$v_{2h+2}$ are not in $G'$, and so on. Thus, all vertices of $C$ are in $G'$. Then again by Lemma~\ref{quasi}, all vertices of $G$ are in $G'$.
It follows that $G'$ is obtained from $G$ by deleting some edges. This contradicts Proposition~\ref{prop:dd}.
   \end{proof}

\end{document}